\documentclass{article}

\usepackage{bm,latexsym,amsfonts,amsmath,amssymb,amsthm,url,graphics,psfrag,amscd,stmaryrd, mathrsfs}
\usepackage[english]{babel}
\usepackage[T1]{fontenc}
\usepackage{float}
\usepackage{graphics}
\usepackage{graphicx}
\usepackage{color}
\usepackage{hyperref}
\usepackage{comment}

\newcommand{\R}{\mathbb{R}}

\newcommand{\N}{\mathbb{N}}

\newcommand{\lin}{\left[\kern-0.15em\left[}
\newcommand{\rin} {\right]\kern-0.15em\right]}
\newcommand{\linf}{[\kern-0.15em [}
\newcommand{\rinf} {]\kern-0.15em ]}
\newcommand{\ilin}{\left]\kern-0.15em\left]}
\newcommand{\irin} {\right[\kern-0.15em\right[}

\newtheorem{theorem}{Theorem}[section]

\newtheorem{definition}[theorem]{Definition}

\newtheorem{proposition}[theorem]{Proposition}

\newtheorem{remark}[theorem]{Remark}

\newcommand{\be}{\begin{equation}}
\newcommand{\ee}{\end{equation}}
\newcommand{\beal}{\begin{align}}
\newcommand{\eal}{\end{align}}
\newcommand{\nn}{\nonumber}

\numberwithin{equation}{section}

\newcommand{\rem}[1]{}

\def\1{{\mathchoice {1\mskip-4mu\mathrm l}      
{1\mskip-4mu\mathrm l}
{1\mskip-4.5mu\mathrm l} {1\mskip-5mu\mathrm l}}}
\newcommand{\indic}[1]{\1_{\{#1\}}}

\newcommand{\invisible}[1]{}

\DeclareSymbolFont{extraup}{U}{zavm}{m}{n}
\DeclareMathSymbol{\varheart}{\mathalpha}{extraup}{86}
\DeclareMathSymbol{\vardiamond}{\mathalpha}{extraup}{87}

\usepackage[english]{babel}

\usepackage[letterpaper,top=2cm,bottom=2cm,left=3cm,right=3cm,marginparwidth=1.75cm]{geometry}

\usepackage{amsmath}
\usepackage{graphicx}

\title{Path space large deviation on a lattice \\ for interacting particle systems}

\author{
  Cristian Giardin\`a\thanks{Department of Physics, Informatics and Mathematics, University of Modena e Reggio Emilia, Modena, Italy. Email: \texttt{cristian.giardina@unimore.it}} \and 
  Tomohiro Sasamoto\thanks{Department of Physics, Institute of Science Tokyo, 2-12-1 Ookayama, Meguro-ku, Tokyo 152-8551, Japan. Email: \texttt{sasamoto@phys.sci.isct.ac.jp}}
}

\begin{document}
\maketitle

\tableofcontents

\begin{abstract}
We investigate the path space large deviation principle (LDP) for a class of interacting particle systems on a fixed finite discrete torus. As an illustrative example, we use the partial exclusion process with maximal occupation of $s$ particles.
We first verify the standard Kurtz-type law of large numbers
for the original and tilted processes, thereby identifying the deterministic
flows around which the large deviations are formulated.
We then establish the path space LDP, with speed $s$, and identify the associated rate function in variational form as the Legendre transform of a nonlinear Hamiltonian. 
The proof relies on martingale techniques, exponential tightness, and a change-of-measure argument based on exponential tilting, which renders atypical trajectories typical under a modified dynamics.

We also explain how the same strategy extends to further interacting particle systems.
The inclusion process illustrates the conservative non-compact finite-jump case.
The harmonic process is qualitatively different: a single transition may move a variable
number of particles, and the limiting Hamiltonian is logarithmic with a non-trivial
domain. This places the harmonic process outside the usual finite-reaction framework
for density-dependent Markov jump processes. Finally, we discuss continuous-state
energy-exchange models, such as the Brownian energy process and an integrable heat-
conduction model, to indicate how the same Hamiltonian large-deviation formalism
appears beyond discrete-state particle systems.

\end{abstract}

\section{Introduction}

Interacting particle systems provide a fundamental framework for connecting microscopic stochastic dynamics with macroscopic deterministic evolution equations. A central theme in the theory is the derivation of hydrodynamic limits and dynamical large deviation principles, which quantify, respectively, the typical and atypical behavior of empirical densities. Classical results for exclusion-type systems show that symmetric dynamics lead, in the diffusive scaling, to the heat equation, while deviations from the hydrodynamic limit are governed by variational rate functionals; see for instance \cite{kipnis2013scaling} and the seminal work \cite{kipnis1989hydrodynamics}. 
The systematic treatment of hydrodynamic limits for interacting particle systems is developed in \cite{demasi2006mathematical}, while the general framework of large deviations is presented in \cite{varadhan2016large}. 

In this paper we consider interacting particle systems on a periodic lattice, whose size $N$ is fixed, and study the large deviations of a {\em scaled} process, while keeping fixed the lattice structure.
The interacting particle systems have a parameter, called ``spin'' $s$.
Denoting the original process as $\{\eta(t), t\ge 0\}$, we define a sequence of scaled processes
$\{X^s(t), t\ge 0\}$ by
\be
X^s(t) = \frac{\eta(s t)}{s}, \qquad\qquad i=1,\ldots, N.
\ee
The behavior of the scaled process in the limit $s\to \infty$ is deterministic.
At finite $s$, the scaled process evolve by following the stochastic differential 
equations
\be
dX^s_i(t) =  \Delta X^s_i(t) dt + dM_i^{s}(t). 
\ee
Here $\Delta$ is the discrete laplacian and $M^{s}(t)$ is a 
martingale vector with quadratic variation of order ${\cal O}(1/s)$, so that the 
typical  fluctuation amplitude of the ${\bf X}^s$ process is of order ${\cal O}(1/\sqrt{s})$.

 The limit $s\to\infty$ is a ``fluid limit'', or small-noise limit. The limiting equation is a
finite-dimensional system of ordinary differential equations on
$\mathbb R^{N}$. 
Path-space large deviations quantify the probability of observing trajectories
which differ from this deterministic fluid limit.
In this set-up, five models on a finite lattice are studied in this paper:
\begin{enumerate}
\item 
the partial exclusion process
\item 
the inclusion process
\item 
the harmonic process
\item 
the Brownian energy process
\item 
the integrable heat conduction model
\end{enumerate}
These are all models of transport:
a conserved quantity (i.e., the number of particles in the first three models, the amount of energy in the last two models) is randomly redistributed among lattice sites.
For a general introduction about these
models and their interpretation as transport models we refer to the recent monograph \cite{GiardinaRedig2026}.

As an illustrative example, we treat at length the case of the  
\emph{partial exclusion process} on the discrete one-dimensional torus. Unlike the simple exclusion process, where each site may contain at most one particle, in the partial exclusion process the occupation variables take values in a finite set of mesh size $1/s$. The scaling parameter $s$ thus controls the granularity of the mass at each site. 
The deterministic limit of the process as
\(s\to\infty\) is the discrete heat equation on the torus. This convergence is a standard instance of Kurtz's theorem for
density-dependent Markov jump processes; see
\cite{Kurtz1970,kurtz1970solutions,kurtz1971limit,EthierKurtz2009}. We nevertheless include the short martingale
proof in the present notation, because the same argument is needed for the
tilted processes used in the large-deviation lower bound. In particular, the
tilted dynamics converges to a deterministic non-linear flow, which is the
trajectory made typical by the change of measure.

Our main result is a large deviation principle for the sequence of processes in path space, equipped with the Skorokhod topology, with speed $s$. The associated rate function is expressed in variational form as the Legendre transform of a nonlinear Hamiltonian determined by the jump structure of the model. We further show that the rate functional admits an equivalent representation in terms of relative entropy between the original path measure and suitably tilted dynamics, in the spirit of \cite{shwartz1995large}. This dual structure clarifies the relation between exponential martingales, change-of-measure techniques, and the Hamilton--Jacobi framework characteristic of dynamical large deviations \cite{varadhan2016large}.
The proof strategy follows the classical program for interacting particle systems \cite{kipnis2013scaling}. Exponential martingales yield the large deviation upper bound for compact sets, which is extended to closed sets via exponential tightness. For the lower bound, a carefully chosen tilting renders prescribed trajectories typical under the modified dynamics, allowing one to transfer concentration properties of the tilted process back to the original measure.

\paragraph{Why large deviations on the lattice?}
Fluctuation theories of diffusive systems have been studied extensively in the literature. At hydrodynamic level, the large deviations of density and current are governed by the so-called Macroscopic Fluctuation Theory \cite{bertini2015macroscopic}, which is believed to be a universal theory
applying to a large class of diffusive systems. So far, this has been established in some  special cases, essentially limited to exclusion-type models, where entropy methods and superexponential estimates are available \cite{kipnis1989hydrodynamics,kipnis2013scaling}. For systems with non-compact state space and reversible measures with
exponential tails, such as the Kipnis--Marchioro--Presutti (KMP) model
\cite{kipnis1982heat},  
the hydrodynamic large-deviation theory is
subtle \cite{BertiniGabrielliLebowitz2005}. In particular, unbounded local energies may lead to singular behaviour
of finite-cost trajectories. A recent result of Heydecker
\cite{kmp-pathologies} shows that, for the KMP process,
the dynamical large-deviation functional admits pathological trajectories with
finite cost. This illustrates that, beyond compact particle systems such as
exclusion, the structure of the path-space rate functional can be considerably
more delicate.
We consider here an easier version of this problem
in which we keep the lattice structure.
The approach we develop  is robust and does not rely on superexponential estimates.

We further show that our strategy can be applied to a large class of interacting particle systems with non-compact state space. In particular, besides exclusion type model, we extend the result to inclusion process and the harmonic process, which provide examples with 
reversible measure given by a discrete Gamma distribution. We also
treat diffusion processes, such as the Brownian energy process, as well as jump process with infinite intensity, such has an integrable model of heat conduction. These extensions highlight the flexibility of the method and show that both compact and non-compact settings can be treated within a unified framework.

We mention a few applications that may justify the interest in the large $s$ large deviations
proved in this paper:
\begin{itemize}

\item
For the partial exclusion process, a formal computation allows to recover from our results
the dynamical large deviation principle of \cite{kipnis1989hydrodynamics}. Our results thus suggest
a new route to establish a dynamical large deviation principle for the density of a broader class of conservative interacting particle systems for which the standard super-exponentail estimate 
is not available.

\item
Furtermore, on the base on the result of this paper, one can also define a discrete version of Macroscopic Fluctuation Theory which keeps the lattice structure.
This was developed in the companion paper \cite{gs1}.

\item 
Finally, for the models that are integrable (such as the harmonic process) one may expect to exploit classical integrability to
compute the large deviation function of
several observables, such as the current
\cite{gmss}.
\end{itemize}

In this paper we restrict to periodic systems defined on the torus. The results can be extended to other geometries, such as systems
defined in infinite volume. See, for instance \cite{landim1992occupation} where a large deviation principle was established for the empirical measure of the
symmetric simple exclusion processes 
on the lattice $\mathbb{Z}^d$
starting from an equilibrium state.

\paragraph{Relation with finite-jump large-deviation theory.}
Large deviations of Markov finite-jump processes is well developed. In particular, a
process with generator
\begin{equation}
\label{network-generator}
    (L^s f)(x)
    =
    s\sum_{r\in R}\lambda_r^s(x)
    \left[
        f\left(x+\frac1s\gamma_r\right)-f(x)
    \right],
\end{equation}
where \(R\) is a fixed finite set of jump channels and
\(\gamma_r\in\mathbb Z^d\) are fixed jump vectors, leads to the Hamiltonian
\be
    H(x,p)
    =
    \sum_{r\in R}
    \lambda_r(x)
    \left(e^{p\cdot\gamma_r}-1\right),
\ee
and this setting has been studied extensively; 
see, for example,
\cite{ShwartzWeiss1995,ShwartzWeiss2005,FengKurtz2006, maes2007and,
AgazziDemboEckmann2018,PattersonRenger2019, monthus2019microscopic,
PardouxSamegniKepgnou2017,gabrielli2020dynamical,AgazziAndreisPattersonRenger2022}.

The partial exclusion process treated below belongs, after the introduction of
hole variables, to this finite-jump class. The inclusion process is also of
finite-jump type, although the non-compactness of its state space requires a
localization argument based on mass conservation. The harmonic process, however,
does not belong to the finite-jump framework. In that model a single transition
may move \(k\) particles, where \(k=1,\ldots,sx_i\). Hence the possible jump
vectors are \(k(e_j-e_i)\), with \(k\) unbounded in the scaling limit. The
limiting Hamiltonian is logarithmic and has a non-trivial finite domain. Thus the large deviation principle for the harmonic process is not
a direct consequence of the finite-reaction large-deviation results cited above.

\paragraph{Organization of the paper.}
Our purpose is twofold. First, we present a self-contained derivation of the
path-space large-deviation principle for the partial exclusion process, in a
form that makes the Hamiltonian structure and the tilting argument transparent.
Second, we show how the same strategy extends beyond the basic compact
finite-jump setting. The inclusion process illustrates the conservative
non-compact finite-jump case, while the harmonic process provides a genuinely
different example with variable-size jumps and a logarithmic Hamiltonian. The
latter feature places the harmonic process outside the standard finite-reaction
framework for density-dependent Markov jump processes.

The paper is organized as follows. In Section~2 we introduce the partial
exclusion process, its tilted dynamics, and the corresponding deterministic
limiting equations. Section~3 proves the law of large numbers for the tilted
process. Section~4 identifies the Hamiltonian and the associated variational
rate functional, and states the path-space large-deviation principle. Sections~5
and~6 contain the proof of the upper and lower bounds, respectively, based on
exponential martingales and a change-of-measure argument. In Section~7 we explain
how the method extends to other models. We first discuss the inclusion process,
which remains a finite-jump conservative system but has non-compact state space.
We then treat the harmonic process, where jumps of variable size lead to a
logarithmic Hamiltonian and require a separate analysis. Finally, we discuss the
continuous-state models and indicate how the same strategy applies there.

\paragraph{Acknowledgment.}
The work of TS has been supported by JSPS KAKENHI Grants No. JP21H04432, No. JP23K22414.
CG acknowledges  support from Istituto Nazionale di Alta Matematica (INdAM) Grant No. CUP E53C23001740001. 
The authors thank T. Bodineau,  K. Mallick, F. Redig, H. Suda for useful discussions and suggestions. 
They are grateful to the Isaac Newton Institute for Mathematical Sciences for its hospitality during the program "New statistical physics in living matter: non-equilibrium states under adaptive control" (SPL), where this work was initiated. Part of this work was performed during the workshop "Rigorous Statistical Mechanics and Related Topics" at the Research Institute for Mathematical Sciences, Kyoto University, the workshops "Hydrodynamics, Fluctuations, and Noise in quantum and classical systems" (code: ICTS/hydrodynamics2025/12) and "Two and a half decades of the Macroscopic Fluctuation Theory 2026" (code: ICTS/MFTHEORY2026/09) at the International Centre for Theoretical Sciences (ICTS), and the workshop "Algebraic Methods in Probability" at the Simons Center for Geometry and Physics (SCGP). We acknowledge them for their support. 
We are very grateful to the anonymous referee for careful reading and a
number of constructive remarks.

\section{The partial exclusion process}

In this section we introduce the partial exclusion process, define a tilted version used in the large deviation analysis, and describe the associated deterministic limits.

\subsection{Model and notations}

Fix once for all a positive integer $N\in\mathbb{N}$ and let $\mathbb{T}_N = \mathbb{Z}/N\mathbb{Z}$ be the discrete torus.
All the quantities that will be considered in this paper depend on the actual value of system size $N$, however in the following we will not write the $N$-dependence to alleviate notation.

Let $\{\eta(t), t\ge0\}$
be the partial exclusion
process, first introduced in \cite{schutz1994non}, taking values in $\{0,1,\ldots, s\}^{\mathbb{T}_N}$ and defined by the generator
\begin{eqnarray}
\label{gen-original}
(Af)({\eta}) 
&=& 
\frac{1}{s^2} \sum_{i\in {\mathbb{T}_N}}  \eta_{i}(s-\eta_{i+1}) \left(f(\eta^{i,i+1})-f(\eta)\right)\nn\\
&+ & \frac{1}{s^2}\sum_{i\in {\mathbb{T}_N}} \eta_{i+1}(s-\eta_i) (f(\eta^{i+1,i})-f(x)).
\end{eqnarray}
Here $\eta^{i,i+1}$ is the configuration obtained from $\eta$
by moving one particle from site $i$ to site $i+1$.

We then consider a sequence of continuous-time Markov process ${\bm X}^s = \{X^{s}(t)\,:\, t\ge 0\}$, where $s\in\N$ is the sequence label,
defined by
\be
X^s(t) = \frac{\eta(s t)}{s}, \qquad\qquad i=1,\ldots, N.
\ee
For a given $s\in\N$ and a given $t\ge 0$, the random vector $X^{s}(t) =(X^s_1(t),X^s_2(t),\ldots, X^s_N(t))$ takes values in $\Omega^{s}=\{0,\frac{1}{s},\frac{2}{s},\ldots,1\}^{\mathbb{T}_N}$.
The generator of the process ${\bm X}^s$, working on functions $f:\Omega^s \to \mathbb{R}$, is given by:
\begin{eqnarray}
\label{gen}
(L^{s}f)({x}) 
&=& 
s \sum_{i\in {\mathbb{T}_N}}  x_{i}(1-x_{i+1}) \left(f(x^{i,i+1})-f(x)\right)\nn\\
&+ & s\sum_{i\in {\mathbb{T}_N}} x_{i+1}(1-x_i) (f(x^{i+1,i})-f(x)).
\end{eqnarray}
Here, each variable $x_i$ represents a rescaled occupation number, i.e. the number of particles at site $i$ divided by $s$. Furthermore, $x^{j,\ell}$ denotes the configuration that is obtained from $x\in\Omega^s$ by moving a mass $\frac{1}{s}$ from site $j$ to site $\ell$, i.e.
\begin{eqnarray}
\label{newconf}
x^{j,\ell}_i =
\begin{cases}
x_i, &  i\neq j,\ell,\\
x_i-\frac{1}{s},\qquad &  i = j,\\
x_i+\frac{1}{s}, & i = \ell.
\end{cases}
\end{eqnarray}
The factor $s$ in the generator corresponds to an acceleration of the dynamics that compensates the jump size $1/s$, ensuring a non-trivial deterministic limit as $s\to\infty$.
As we will work only with the scaled process ${\bm X}^s$, we shall use for it the same name of the original process ${\bm \eta}$, i.e. it will be called the {\em partial exclusion process with parameter (spin) value $s$}.
\footnote{
The name ``spin'' refers here to the identification of the Markov generator with the Hamiltonian of quantum spin chain (see \cite{schutz1994non}).
Our parameter $s$ is actually twice the spin value of the quantum chain.}.

For a fixed time interval $[0,T]$, we denote by $E^s = D([0,T];\Omega^s)$ the space of right continuous trajectories on $\Omega^s$ with left limits. Elements of $E^s$ are denoted by ${\bm x}^s=\{x^s(t)\,:\; t\in[0,T]\}$. 
We next discuss the initial condition of the process, which we call the ``local equilibrium initial condition''.
Given a vector $p=(p_1,\ldots,p_N)\in[0,1]^{N}$, we define the inhomogeneous product measures $\nu^s_{p}$  with mass function
\be
\label{initial-excl}
\nu^s_{p}(x) = \bigotimes_{i\in {\mathbb T}_N} {s \choose s x_i  } p_i^{s x_i }(1-p_i)^{s(1-x_i) }, \qquad\qquad x \in \Omega^s.
\ee
We denote
by $\mathbb{P}^{s}_{p}$ the path space measure of the process ${\bm X}^s$ initialized from configurations that are distributed according to $\nu^s_{p}$.
Defining $\bar{\mathsf{p}} = (\mathsf{p},\ldots,\mathsf{p})$, the homogeneous product measure $\nu_{\bar{\mathsf{p}}}$ is reversible for any  $0<\mathsf{p} < 1$.
Thus the local equilibrium initial conditions corresponds to starting from an inhomogenous version of the reversible measure, where
the marginal at each site $i$ is a scaled  binomial distribution with mean $p_i$.

We shall prove that, as $s\to\infty$, a trajectory of the partial exclusion process concentrates (in probability) on a deterministic flow ${\bm X} = \{X(t)\, : \, t\ge 0\}$ where
${X}(t)= (X_1(t),X_2(t),\ldots, X_{N}(t))$  takes values on $\Omega=[0,1]^{\mathbb{T}_N}$. The flow is defined as the solution of the system of differential equations
\begin{eqnarray}
\label{flow}
\frac{dX_i(t)}{dt} &=& X_{i-1}(t)- 2 X_i(t) + X_{i+1}(t), \qquad\qquad i\in\mathbb{T}_N,  
\end{eqnarray}
with prescribed initial condition $X(0)\in\Omega$.
Equivalently, the flow is generated  by the first-order differential operator 
\begin{eqnarray}
\label{generator-flow}
(Lf)({x}) 
&=& 
- \sum_{i\in {\mathbb{T}_N}}  (x_{i}-x_{i+1})
\left(\frac{\partial}{\partial x_i}-\frac{\partial}{\partial x_{i+1}}\right)f(x).
\end{eqnarray}

\subsection{The tilted partial exclusion process}
To prove the large deviations lower bound, we introduce a tilted process that modifies the jump rates so that a prescribed atypical trajectory becomes typical under the new dynamics.
The tilted partial exclusion process, denoted by
${\bm X}^{s,h} = \{X^{s,h}(t)\, : \, t\ge 0\}$, is the time-inhomogeneous Markov process taking values in $\Omega^s$ and generated by
\begin{align}
\label{gen-tilted}
(L^{s,h}f)({x}) 
&=~~ 
s \sum_{i\in {\mathbb{T}_N}}  x_{i}(1-x_{i+1}) e^{h_{i+1}(t)-h_i(t)} \left(f(x^{i,i+1})-f(x)\right)\nn\\
&\quad + s\sum_{i\in {\mathbb{T}_N}} x_{i+1}(1-x_i) e^{h_i(t)-h_{i+1}(t)}(f(x^{i+1,i})-f(x)).
\end{align}
Here ${h} : \mathbb{R} \to \mathbb{R}^{\mathbb{T}_N}$
is a continuous function which is used to tune the tilting of the jump rates and makes the process time-inhomogeneous. For a fixed interval of time $[0,T]$ we denote by $\mathbb{P}^{s,h}_{p}$ the law of the process $\bm{X}^{s,h}$ with initial distribution $\nu^s_p$.

The large $s$ limit of the tilted partial exclusion process will be the deterministic flow ${\bm X} = \{X(t)\, : \, t\ge 0\}$ taking values in $\Omega= [0,1]^{\mathbb{T}_N}$ and solving 
the equations
\begin{eqnarray}
\label{flow-tilted}
\frac{dX(t)}{dt} 
=
F^{h}(X(t)),
\end{eqnarray}
with initial conditions $X(0)\in\Omega$, where $F^{h}:\Omega \to \R^{\mathbb{T}_N}$ is defined by
\begin{align}
\label{vector-b}
F_i^{h}({x})
&=~~ 
x_{i-1}(1-x_{i})e^{h_{i}(t)-h_{i-1}(t)} - x_{i}(1-x_{i-1})e^{h_{i-1}(t)-h_{i}(t)} \nn\\
&\quad -x_i(1-x_{i+1})e^{h_{i+1}(t)-h_{i}(t)} + x_{i+1}(1-x_i)e^{h_i(t)-h_{i+1}(t)}.
\end{align}
The first two terms correspond to interactions with site $i-1$, while the last two correspond to interactions with site $i+1$.
Equivalently, the limiting flow is generated by the first-order operator 
\begin{eqnarray}
\label{generator-flow-tilted}
(Lf)({x}) 
&=& 
- \sum_{i\in {\mathbb{T}_N}}  \left(x_{i}(1-x_{i+1}) e^{h_{i+1}(t)-h_i(t)} -
x_{i+1}(1-x_i) e^{h_i(t)-h_{i+1}(t)}\right)
\left(\frac{\partial}{\partial x_i}-\frac{\partial}{\partial x_{i+1}}\right)f(x).\nn\\
\end{eqnarray}
This representation highlights the conservative structure of the dynamics.

\subsection{Remarks}
We conclude this section with two remarks.
\begin{itemize}
\item[(i)] {\em Recovery of the original process.} If $h\equiv 0$,  then the tilted process reduces to the original partial exclusion process, namely  $\bm{X}^{s,0} = \bm{X}^{s}$,
and the limiting deterministic dynamics reduces to the discrete diffusion equation \eqref{flow}.
\item[(ii)] {\em Relation with macroscopic PDEs.}
Eq. \eqref{flow} can be interpreted as a lattice discretization of the heat equation for the function $\rho(\theta,t)$ on the continuos torus $\mathbb{T} = {\mathbb R}/{\mathbb Z}$
\be
\frac{\partial \rho}{\partial t} = \frac{\partial^2 \rho}{\partial \theta^2}.
\ee
Formally, the correspondence is obtained by setting $X_i(t) \simeq \rho(i/N,t/N^2)$. 
This correspondence is heuristic and reflects the structure of a discrete Laplacian; no scaling limit in \(N\) is considered in this work.
Similarly, Eq. \eqref{flow-tilted} is a lattice discretization of the equation 
\be
\frac{\partial \rho}{\partial t} = \frac{\partial^2 \rho}{\partial \theta^2} - \frac{\partial}{\partial \theta}
\left(\rho(1-\rho)\frac{\partial H}{\partial \theta}\right)
\ee
with the additional identification
$h_i(t) \simeq H(i/N,t/N^2)$.
The additional term corresponds to a weak asymmetry induced by the tilt \cite{kipnis1989hydrodynamics}.
\end{itemize}

\section{Law of large numbers (Kurtz's theorem)}
\label{section-lln}


The convergence results in this section are standard consequences of Kurtz's
theory for density-dependent Markov jump processes
\cite{Kurtz1970,EthierKurtz2009}. We include the short proof for two reasons:
first, to fix notation and identify the deterministic vector fields used later;
second, because the lower bound in the large-deviation principle uses tilted
processes for which the corresponding deterministic flow must be made explicit.

\begin{theorem}[Law of large numbers]
\label{th-lln}
Consider the tilted partial exclusion process $\bm{X}^{s,h}$ started
from the initial measure
$\nu_p^s$ in \eqref{initial-excl}.
Then, for every $T>0$ and every $\epsilon >0$ we have
\begin{equation} 
\label{main-kurtz}
\lim_{s\to\infty}\mathbb{P}^{s,h}_p  (\sup_{t\in [0,T]} |{X}^{s,h}(t)-{X}(t)| > \epsilon) =0, 
\end{equation}
where $\bm{X}$ is the deterministic flow solving \eqref{flow-tilted} with initial condition
${X}(0) = {p}$.
\end{theorem}

\begin{proof}
Since $\bm{X}^{s,h}$ is a Markov process then, defining $f_i(x) = x_i$ for $i\in\mathbb{T}_N$,  we have that  $\bm{M}^{s,h}$ 
is a  martingale vector with components
\be
{M}_i^{s,h}(t) = X_i^{s,h}(t) - X_i^{s,h}(0) - \int_{0}^t L^{s,h} f_i({X^{s,h}}(u)) du.
\ee
Therefore we can write
\be
\label{x-s-vector}
X_i^{s,h}(t)  =   X_i^{s,h}(0) + \int_{0}^t L^{s,h} f_i({X^{s,h}}(u)) du
+ {M}_i^{s,h}(t).
\ee
Similarly, for the flow $\bm{X}$ we can write
\be
\label{x-vector}
X_i(t)  =   X_i(0) + \int_{0}^t L f_i(X(u)) du.
\ee
Recalling the definition of the generators $L^{s,h}$ in \eqref{gen-tilted} and $L$ 
in \eqref{generator-flow-tilted} a direct computation shows that
\begin{eqnarray}
\label{L-action-on-f_i}
L^{s,h}f_i({x}) &=&  F^{h}_i({x}),\qquad\qquad x\in\Omega^s,\nn\\
Lf_i({x}) &=& F^{h}_i({x}),\qquad\qquad x\in\Omega,
\end{eqnarray}
where the function
$F^{h}$ is defined in \eqref{vector-b}. Note that the right-hand sides of the above equations are the same implying that, on their common domain, $L^{s,h}f_i=Lf_i$
for all values of $s$.
All in all,
combining together \eqref{x-s-vector}, \eqref{x-vector}
and \eqref{L-action-on-f_i}, we can write
\be
{X}^{s,h}(t)-{X}(t) =
{X}^{s,h}(0)-{X}(0) + 
\int_{0}^t [{F}^{h}({X}^{s,h}(u)) - {F}^{h}({X}(u))] du +  {M}^{s,h}(t) 
\ee
where the equality is understood component-wise. Taking norms and using triangular inequality, we get
\be
\label{triang-inequality}
|{X}^{s,h}(t)-{X}(t)| \le
|{X}^{s,h}(0)-{X}(0)| + 
\int_{0}^t |{F}^{h}({X}^{s,h}(u)) - {F}^{h}({X}(u))| du +  |{M}^{s,h}(t)|.
\ee
Since $h \in C([0,T];\mathbb{R}^N)$, it is bounded. 
Hence the quantities
$
e^{\pm (h_{i+1}(t)-h_i(t))}
$
are uniformly bounded for $t \in [0,T]$. Therefore, since the vector field $F^h$ is polynomial in $x$ with bounded coefficients, then it is
globally Lipschitz.
Hence there exists $K>0$ such that
\[
\lvert F^h(x) - F^h(y) \rvert \le K \lvert x - y \rvert,
\qquad \forall x,y \in \Omega,\ \forall t \in [0,T].
\]
As a consequence, we have that
\be
\label{ineq-after-lipschitz}
|{X}^{s,h}(t)-{X}(t)| \le
|{X}^{s,h}(0)-{X}(0)| + 
K \int_{0}^t |{X}^{s,h}(u) - {X}(u)| du +  |{M}^{s,h}(t)|.
\ee
To proceed further we recall Gr\"onwall's inequality (see for instance \cite{ethier2009markov}, p. 498): if one has a bounded function $y:[0,T]\to \R$ such that for all $t\in[0,T]$
\be
y(t) \le A + B \int_{0}^t y(u)du
\ee
where $A$ and $B$ are some positive constants (which might depend on $T$), then one has for all $t\in[0,T]$
\be
y(t) \le A e^{B T}.
\ee
Taking the supremum over \(t\in[0,T]\) in \eqref{ineq-after-lipschitz} and applying
Gronwall's inequality gives
\[
    \sup_{t\in[0,T]} |X^{s,h}(t)-X(t)|
    \leq
    e^{KT}
    \left(
        |X^{s,h}(0)-X(0)|
        +
        \sup_{t\in[0,T]} |M^{s,h}(t)|
    \right).
\]
Consequently,
\begin{eqnarray} 
\label{eq:intermediate}
&&\mathbb{P}_{p}^{s,h}   (\sup_{t\in [0,T]} |{X}^{s,h}(t)-{X}(t)| > \epsilon) 
\\
&&\le 
\mathbb{P}_{p}^{s,h}   (|{X}^{s,h}(0)-{X}(0)| > \frac{\epsilon}{2} e^{-KT}) +\mathbb{P}_{p}^{s,h}   ( \sup_{t\in [0,T]} |{M}^{s,h}(t)|> \frac{\epsilon}{2} e^{-KT}). \nonumber
\end{eqnarray}
Recalling \eqref{initial-excl}, it is easy to see that the first term on the r.h.s. vanishes as $s\to\infty$ when $X(0) = p$. It remains to control the martingale term, namely to prove
\eqref{main-kurtz} it is enough to show that
\begin{equation} 
\label{M-ineq}
\lim_{s\to\infty}\mathbb{P}_{p}^{s,h}   (\sup_{t\in [0,T]} |{M}^{s,h}(t)|> \frac{\epsilon}{2} e^{-KT}) =0.
\end{equation}
To this aim, we compute the quadratic covariation defined as
\begin{eqnarray}
\langle M_j^{s,h}, M_k^{s,h}\rangle(t) 
&=& 
\int_{0}^t \Gamma(f_j,f_k)({X}^{s,h}(u))du
\end{eqnarray}
where
\begin{align}
\label{covar}
\Gamma(f_j,f_k)({x})
&=~~ 
s \sum_{i\in {\mathbb{T}_N}}  x_{i}(1-x_{i+1}) e^{h_{i+1}(t)-h_i(t)} \left(f_j(x^{i,i+1})-f_j(x)\right)\left(f_k(x^{i,i+1})-f_k(x)\right)\nn\\
&\quad + s\sum_{i\in {\mathbb{T}_N}} x_{i+1}(1-x_i) e^{h_i(t)-h_{i+1}(t)}
\left(f_j(x^{i,i+1})-f_j(x)\right)\left(f_k(x^{i,i+1})-f_k(x)\right).\nn\\
\end{align}
Each jump changes \(f_i(x)=x_i\) by at most \(1/s\), while jump rates are of order \(s\). Therefore the carré du champ satisfies
\[
\Gamma(f_i,f_i)(x) = O\left(\frac{1}{s}\right),
\]
uniformly in \(x\), which implies
\be
\lim_{s\to\infty} \langle M_i^{s,h},M_j^{s,h}\rangle(t) = 0.
\ee
Next, applying Doob's inequality and 
using the isometry for square-integrable martingales we conclude that, for each $i\in [N]$,
\begin{eqnarray}
\label{doob}
\lim_{s\to\infty}\mathbb{P}_{p}^{s,h}(\sup_{t\in [0,T]} |{M}_i^{s,h}(t)| > \frac{\epsilon}{2} e^{-KT}) \le 
\lim_{s\to\infty} 4\frac{\mathbb{E}[(M_i^{s,h}(T))^2]}{\epsilon^2 e^{-2KT}} 
=
\lim_{s\to\infty} 4\frac{\mathbb{E}\left[\langle M_i^{s,h},M_i^{s,h}\rangle(T)]\right]}{\epsilon^2e^{-2KT}} =0.\nonumber\\
\end{eqnarray}
This implies  \eqref{M-ineq} and therefore proves \eqref{main-kurtz}.
\end{proof}

The special structure of the initial measure $\nu^{s}_p$ in \eqref{initial-excl} did not play any particular role
in the proof above. The choice of the initial measure only matters to show that the first term on the right hand side of \eqref{eq:intermediate}
 vanishes in the large $s$ limit. Therefore the hypothesis of Theorem \ref{th-lln} concerning the initial state can be considerably relaxed, as the next theorem shows.

\begin{theorem}[Law of large numbers, generic initial conditions]
\label{th-lln-gen}
Consider the tilted partial exclusion process $\bm{X}^{s,h}$ started from the initial measure ${\nu^{s}}$ and denote by $\mathbb{P}_{\nu^{s}}^{s,h}$
its path space probability measure.
For a given $x\in\Omega$, 
assume that the initial distributions $(\nu^{s})_{s\in\mathbb{N}}$ are such that 
\begin{equation}
\label{assumption-initial}
X^{s,h}(0) \xrightarrow{{\mathbb{P}_{\nu^{s}}}} x
\end{equation}
that is,
\begin{equation}
\lim_{s\to\infty} \mathbb{P}_{\nu^{s}}\big(|X^{s,h}(0)-x|> \epsilon\big) = 0
\qquad \forall \epsilon > 0.
\label{eq:initial-convergence}
\end{equation}
Then, for every $T>0$ and every $\epsilon >0$ we have
\begin{equation} 
\label{main-kurtz-general-init-cond}
\lim_{s\to\infty}\mathbb{P}_{\nu^{s}}^{s,h}  (\sup_{t\in [0,T]} |{X}^{s,h}(t)-{X}(t)| > \epsilon) =0, 
\end{equation}
where $\bm{X}$ is the deterministic flow solving \eqref{flow-tilted} with initial condition
${X}(0) = {x}$.
\end{theorem}
\begin{proof}
The proof is identical to that of Theorem \ref{main-kurtz}, replacing convergence to $p$ with convergence to $x$.
\end{proof}

\section{Large deviations}

Large deviation principle (around the hydrodynamical limit)  for the empirical density of the exclusion process has been proved in 
\cite{kipnis1989hydrodynamics}, see also    \cite{casini2024large} for the generalization to the multi-species case. 
In the present section we instead establish a fixed-lattice
path-space large-deviation principle for the partial exclusion process in the
fluid limit \(s\to\infty\).

\subsection{Rate function}

We first define the rate function as the sum of a dynamical and an initial cost.

\paragraph{Dynamical part.}

For $x \in \Omega$ and $w \in \mathbb{R}^{\mathbb{T}_N}$, we define  the Hamiltonian associated with the partial exclusion process:
\begin{equation}
H(x,w)
=
\sum_{i \in \mathbb{T}_N}
\Big(
x_i(1-x_{i+1})(e^{w_{i+1}-w_i}-1)
+
x_{i+1}(1-x_i)(e^{w_i-w_{i+1}}-1)
\Big).
\label{hamiltonian}
\end{equation}
Then, given $\bm{x} \in D([0,T],\Omega)$, we define the dynamical cost as
\begin{equation}
I_{\mathrm{dyn}}(\bm{x})
=
\sup_{{h} \in C^1([0,T];\mathbb{R}^{\mathbb{T}_N})}
\left[
h(T)\cdot x(T) - h(0)\cdot x(0)
-
\int_0^T \big(\dot h(t)\cdot x(t) - H(x(t),h(t))\big)\,dt
\right].
\label{eq:rate-global}
\end{equation}
If $\bm{x}$ is absolutely continuous, an integration by parts gives
\be
I_{\mathrm{dyn}}(\bm{x})
=
\sup_{{h} \in C^1([0,T];\mathbb{R}^{\mathbb{T}_N})}
\int_0^T \Big(h(t)\cdot \dot x(t) - H(x(t),h(t))\Big)\,dt.
\ee
Since the integrand is local in time and $H(x,\cdot)$ is finite and convex,
a standard interchange argument 
(see for instance \cite{dembo2009large}, section 5.1)
allows one to exchange the supremum and the time integral. Therefore 
the dynamical part can be cast in the form of an action functional, i.e.
\begin{equation}
I_{\mathrm{dyn}}(\bm{x})
=
\begin{cases}
\displaystyle \int_0^T L(x(t),\dot x(t))\,dt,
& \text{if } \bm{x} \text{ is absolutely continuous}, \\[1ex]
+\infty,
& \text{otherwise},
\end{cases}
\label{eq:rate-dyn}
\end{equation}
where the Lagrangian is
\begin{equation}
L(x,v)
=
\sup_{w \in \mathbb{R}^{\mathbb{T}_N}}
\big\{ w \cdot v - H(x,w) \big\}.
\label{eq:Lagrangian}
\end{equation}

\medskip

\paragraph{Static part.}

Let $(\nu^s_p)$ be the initial distribution in \eqref{initial-excl}.
We define the cost of the
local equilibrium initial condition as
\begin{equation}
I_0(x)
=
\begin{cases}
\sum_{i \in \mathbb{T}_N}
\left(
x_i \log \frac{x_i}{p_i}
+
(1-x_i)\log \frac{1-x_i}{1-p_i}
\right), & x\in\Omega,\\
+\infty, & \text{otherwise}.
\end{cases}
\label{eq:initial-cost}
\end{equation}
Thus $I_0$ coincides with the large deviation rate
function of the Binomial$(s,p)$ distribution scaled by $1/s$, which is in turn given
by the relative entropy between Bernoulli distributions with parameters
$x$ and $p$.

\medskip

\paragraph{Full rate function.}

The rate function $I : D([0,T],\Omega) \to [0,+\infty]$ is defined by
\begin{equation}
I(\bm{x})
=
I_0(x(0)) + I_{\mathrm{dyn}}(\bm{x}).
\label{rate-fct-dem}
\end{equation}
We recall that a good rate function is, by definition, a lower semicontinuous
function whose level sets are compact. For the functional \(I\) defined above,
lower semicontinuity and compactness of the level sets follow from the standard
properties of action functionals associated with finite-dimensional
density-dependent jump processes; see, for instance,
\cite[Chapter~8]{FengKurtz2006} and \cite[Chapter~5]{ShwartzWeiss1995}.
Indeed, the Hamiltonian \(H(x,\cdot)\) is convex, the Lagrangian \(L\) is its
Legendre transform, and bounded action gives compact sets of paths in
\(D([0,T],\Omega)\).

\subsection{Main result}

\begin{theorem}[Large deviation principle]
\label{main-theorem}
The sequence of partial exclusion processes $(\bm{X}^s)_{s\in\mathbb{N}}$ with  initial distributions $(\nu^s_p)_{s\in \N}$ in \eqref{initial-excl}, satisfies a large deviation principle on $D([0,T],\Omega)$ with speed $s$ and good rate function $I$ defined in \eqref{rate-fct-dem}.
More precisely,  we have the following:
\begin{enumerate}
\item
(Upper bound). For every closed set ${\cal C}$,
\[
\limsup_{s\to\infty} \frac{1}{s}\log \mathbb{P}^s_{p}(\bm{X}^s \in {\cal C})
\le -\inf_{\bm{x}\in {\cal C}} I(\bm{x}).
\]
\item
(Lower bound). For every open set ${\cal O}$,
\[
\liminf_{s\to\infty} \frac{1}{s}\log \mathbb{P}^s_{p}(\bm{X}^s \in {\cal O})
\ge -\inf_{\bm{x}\in {\cal O}} I(\bm{x}).
\]
\end{enumerate}
\end{theorem}

\subsection{Proof strategy}
Following \cite{kipnis1989hydrodynamics, varadhan2016large}, we prove Theorem \ref{main-theorem} by implementing a strategy based on martingales and tilting. In particular, the proof will be organized as follows: 
\begin{itemize}
\item[(i)]  in Section \ref{sec-upper}, the large deviation upper bound for closed sets is established by using martingales. We
first derive in Proposition \ref{upper-prop} the large deviation upper bound for compact sets by considering an exponential martingale; in a second stage (Proposition \ref{prop-tightness}), we extend the large deviation upper bound to closed sets by showing the exponential tightness property.
\item[(ii)] in Section \ref{sec-lower}, the large deviation lower bound for open sets is established by using the tilted process
$\bm{X}^{s,h}$.
The tilting $h$ is chosen such that
a trajectory which is rare with respect to the path space measure of the original process becomes typical for the tilted process.
To exploit tilting, we first provide a rewriting
of the rate function in terms of relative entropy
in Proposition \ref{rate2}. We then conclude the proof of the lower bound in
Proposition \ref{lower-prop}.
\end{itemize}

\section{Upper bound}
\label{sec-upper}

\begin{proposition}[Upper bound for compact sets]
\label{upper-prop}
For every compact set ${\cal K}\subset D([0,T],\Omega)$ we have
\begin{equation}  
\label{eq:upper-det}
\limsup_{s\to\infty}\frac{1}{s}\log\mathbb{P}_p^s (\bm{X}^s\in {\cal K}) \le -\inf
_{\substack{\bm{x} \in {\cal K}\\ }}
I(\bm{x}). 
\end{equation}
\end{proposition}
\begin{proof}
We use multiplicative functionals. Given
$f:\mathbb R_+\times \Omega^s\to \mathbb R$, the following is a martingale \cite{ethier2009markov}:
\be
M^s(t)
=
\frac{f(t,{X}^s(t))}{f(0,{X}^s(0))}
\exp\left\{
-\int_0^t
\frac{(\partial_u+L^s)f(u,{X}^s(u))}{f(u,{X}^s(u))}\,du
\right\}.
\ee
Therefore, for any vector function
$h\in C^1([0,T];\mathbb R^{\mathbb T_N})$, by choosing
\be
f(t,x)=\exp\left\{s\sum_{i\in\mathbb T_N}h_i(t)x_i\right\},
\ee
we get the exponential martingale
$\bm{Z}^{s,h}=\{Z^{s,h}(t):0\le t\le T\}$ defined by
\be
\label{Z_T}
Z^{s,h}(T)=\exp\{s I_{\rm dyn}^h(\bm{X}^s)\},
\ee
where
\be
\label{Ih}
I_{\rm dyn}^h(\bm{x})
=
h(T)\cdot {x}(T)-h(0)\cdot {x}(0)
-\int_0^T
\left[
\dot h(u)\cdot {x}(u)-H({x}(u),h(u))
\right]du
\ee
and $H(x,w)$ is defined in \eqref{hamiltonian}.
We start by considering an open set
${\cal O}\subset D([0,T],\Omega)$. In order to include the cost of the
random initial condition, we also fix $a\in \mathbb R^{\mathbb T_N}$ and write
\be
\mathbb{P}_p^s(\bm{X}^s\in {\cal O})
=
\mathbb{E}_p^s\left[
\mathbf 1_{\{\bm{X}^s\in {\cal O}\}}
Z^{s,h}(T)e^{s a\cdot {X}^s(0)}
\cdot \frac{e^{-s a\cdot {X}^s(0)}}{Z^{s,h}(T)}
\right].
\ee
Plugging \eqref{Z_T} in the denominator, we obtain
\be
\mathbb{P}_p^s(\bm{X}^s\in {\cal O})
\le
\mathbb{E}_p^s
\left[
Z^{s,h}(T)e^{s a\cdot {X}^s(0)}
\right]
\exp\left\{
s\sup_{\bm{x}\in O}
\left[
-I_{\rm dyn}^h(\bm{x})-a\cdot {x}(0)
\right]
\right\}.
\ee
Since $\bm{Z}^{s,h}$ is a mean-one martingale conditionally on the initial
configuration, 
\be
\mathbb{E}_p^s
\left[
Z^{s,h}(T)e^{s a\cdot {X}^s(0)}
\right]
=
\mathbb{E}_{\nu_p^s}
\left[
e^{s a\cdot {X}^s(0)}
\right].
\ee
Under $\nu_p^s$, the variables $s{X}_i^s(0)$ are independent binomial random
variables with parameters $(s,p_i)$. Hence
\be
\mathbb{E}_{\nu_p^s}
\left[
e^{s a\cdot {X}^s(0)}
\right]
=
\prod_{i\in\mathbb T_N}
\left(1-p_i+p_i e^{a_i}\right)^s
=
\exp\left\{
s\sum_{i\in\mathbb T_N}
\log(1-p_i+p_i e^{a_i})
\right\}.
\ee
Therefore
\be
\frac1s\log \mathbb{P}_p^s(\bm{X}^s\in {\cal O})
\le
\sup_{\bm{x}\in {\cal O}}
\left[
-I_{\rm dyn}^h(\bm{x})-a\cdot {x}(0)
+\sum_{i\in\mathbb T_N}\log(1-p_i+p_i e^{a_i})
\right].
\ee
Since this bound holds for every $h$ and every $a$, we get
\be
\limsup_{s\to\infty}
\frac1s\log \mathbb{P}_p^s(\bm{X}^s\in {\cal O})
\le
\inf_{\substack{h\in C^1([0,T];\mathbb R^{\mathbb T_N})\\
a\in\mathbb R^{\mathbb T_N}}}
\sup_{\bm{x}\in O}
\left[
-I_{\rm dyn}^h(\bm{x})-a\cdot {x}(0)
+\sum_{i\in\mathbb T_N}\log(1-p_i+p_i e^{a_i})
\right].
\ee

We now use Lemma 3.3 in Appendix 2 of \cite{kipnis2013scaling}, which allows one
to obtain a large deviation upper bound for compact sets from the previous
open-set bound by replacing the $\inf\sup$ by $\sup\inf$. Thus, for every compact
set ${\cal K}\subset D([0,T],\Omega)$,
\be
\begin{aligned}
\limsup_{s\to\infty}
\frac1s\log \mathbb{P}_p^s(\bm{X}^s\in {\cal K})
&\le~~ 
\sup_{\bm{x}\in {\cal K}}
\inf_{\substack{h\in C^1([0,T];\mathbb R^{\mathbb T_N})\\
a\in\mathbb R^{\mathbb T_N}}}
\left[
-I_{\rm dyn}^h(\bm{x})-a\cdot {x}(0)
+\sum_{i\in\mathbb T_N}\log(1-p_i+p_i e^{a_i})
\right]  \\
&\quad=
-\inf_{\bm{x}\in {\cal K}}
\sup_{\substack{h\in C^1([0,T];\mathbb R^{\mathbb T_N})\\
a\in\mathbb R^{\mathbb T_N}}}
\left[
I_{\rm dyn}^h(\bm{x})+a\cdot {x}(0)
-\sum_{i\in\mathbb T_N}\log(1-p_i+p_i e^{a_i})
\right].
\end{aligned}
\ee
The supremum separates into the dynamical and static parts: 
on one hand, recalling
the definitions 
\eqref{eq:rate-global}
and \eqref{Ih}, we have
\be
\sup_{h} I_{\rm dyn}^h(\bm{x})=I_{\rm dyn}(\bm{x}).
\ee
On the other hand,
recalling the definition
\eqref{eq:initial-cost}
we have
\be
\sup_{a\in\mathbb R^{\mathbb T_N}}
\left\{
a\cdot {x}(0)
-\sum_{i\in\mathbb T_N}\log(1-p_i+p_i e^{a_i})
\right\}
=
I_0({x}(0)).
\ee
Consequently,
\be
\limsup_{s\to\infty}
\frac1s\log \mathbb{P}_p^s(\bm{X}^s\in {\cal K})
\le
-\inf_{\bm{x}\in {\cal K}}
\left[
I_{\rm dyn}(\bm{x})+I_0({x}(0))
\right] = -\inf_{\bm{x}\in {\cal K}}I(\bm{x}).
\ee
\end{proof}
To extend the large deviation upper bound 
from compact sets to closed sets it is enough to prove exponential tightness of the measure $\mathbb{P}^s_{p}$. This is the context of the next proposition.
Before proving exponential tightness, let us point out that compact containment
is automatic here, since \(\Omega=[0,1]^{\mathbb T_N}\) is compact. For the
non-compact conservative models considered in Section~7, the same role is
played by conservation of the total mass or total energy, which confines the
process to a compact simplex. The proof below therefore focuses on the second
ingredient in path-space exponential tightness, namely the exponential control
of time oscillations; see also \cite[Theorem~4.4]{FengKurtz2006}.

\begin{proposition}[Exponential tightness]
\label{prop-tightness}
For any $n\in\N$ there exists a compact set ${\cal K}_n\subset D([0,T],\Omega)$ such that
\be
\label{tight:eq}
\limsup_{s\to\infty} \frac{1}{s} \ln\mathbb{P}^s_{p}(\bm{X}^s\notin{\cal K}_n) \le -n
\ee
\end{proposition}
\begin{proof}
We adapt the approach in \cite{kipnis2013scaling}, pag. 271. Namely, we proceed in two steps: in the first step
it is proved that for all vectors ${w}\in\R^N$
and for all $\epsilon>0$
\be
\label{step1}
\lim_{\delta\to 0} \limsup_{s\to\infty} \frac{1}{s} \ln\mathbb{P}_{p}^s(\sup_{|t-t'|\le \delta}|{X}^{s}(t)\cdot {w} -{X}^s(t')\cdot {w}|>\epsilon) = -\infty.
\ee
In the second step, the compact set ${\cal K}_n$ is constructed explicitly.

\paragraph{First step.} Since
\be
\left\{\sup_{|t-t'|\le \delta}|{X}^{s}(t)\cdot {w} -{X}^s(t')\cdot {w}|>\epsilon \right\}
\subset
\bigcup_{k=0}^{\lceil T\delta^{-1} \rceil}
\left\{\sup_{k\delta\le t<(k+1)\delta}|{X}^{s}(t)\cdot {w} -{X}^s(k \delta)\cdot {w}|>\frac{\epsilon}{4} \right\}
\ee
it follows that
\begin{eqnarray}
&&\lim_{\delta\to 0} \limsup_{s\to\infty} \frac{1}{s} \ln\mathbb{P}_{p}^s\left(\sup_{|t-t'|\le \delta}|{X}^{s}(t)\cdot {w} -{X}^s(t')\cdot {w}|>\epsilon\right)\nn\\
&&\le
\lim_{\delta\to 0} \limsup_{s\to\infty} \frac{1}{s} \ln 
\sum_{k=0}^{\lceil T\delta^{-1} \rceil}
\mathbb{P}_{p}^s\left(\sup_{k\delta\le t<(k+1)\delta}|{X}^{s}(t)\cdot {w} -{X}^s(k\delta)\cdot {w}|>\frac{\epsilon}{4}\right)\nn\\
&&
\le
\lim_{\delta\to 0}\limsup_{s\to\infty}
\frac1s \ln
\left[
\left\lceil T\delta^{-1} \right\rceil
\sup_{y\in \Omega^s}
\mathbb{P}^s\Big(
\sup_{0\le t<\delta}
|X^s(t)\cdot w - X^s(0)\cdot w| > \epsilon/4
\,\Big|\, X^s(0)=y
\Big)
\right],\nn\\
&&
\end{eqnarray}
where we used the Markov property and introduced, for $y\in\Omega^s$,
the notation $\mathbb{P}^s(\cdot \mid X^s(0)=y)$ for the law of the process
started from the deterministic configuration $y$.
It remains to estimate, uniformly in $y\in\Omega^s$,
\be
\mathbb{P}^s\Big(
\sup_{0\le t<\delta}
|X^s(t)\cdot w - X^s(0)\cdot w| > \epsilon
\,\Big|\, X^s(0)=y
\Big).
\ee
Now, since we are considering an arbitrary vector ${w}\in\R^N$, we can neglect the absolute value. 
Moreover, for any $a > 0$, consider the mean 1 exponential martingale
\begin{eqnarray}
Z^{a,w}(t) 
&=& 
\exp\Big\{s
\Big[
a {X}^s(t) \cdot w - a  {X}^s(0)\cdot w 
- A^{a,w}(t)\Big]\Big\}
\end{eqnarray}
where
\be
A^{a,w}(t) 
=
\int_{0}^t\sum_{i\in\mathbb{T}_N}\Big({X}^s_i(u)(1-{X}^s_{i+1}(u)) (e^{a(w_{i+1}-w_i)}-1) + {X}^s_{i+1}(u)(1-{X}^s_i(u)) (e^{a(w_i-w_{i+1})}-1)\Big)du.
\ee
Then we have
\begin{eqnarray}
&&\mathbb{P}^s(\sup_{0\le t<\delta}\left({X}^{s}(t)\cdot {w} -{X}^s(0)\cdot {w}\right)>{\epsilon}\,\Big|\, X^s(0)=y)\nn\\
&&=
\mathbb{P}^s(\sup_{0\le t<\delta}\left(a{X}^{s}(t)\cdot {w} -a{X}^s(0)\cdot {w}\right)>{a\epsilon}\,\Big|\, X^s(0)=y)\nn\\
&&=
\mathbb{P}^s\left(\sup_{0\le t<\delta}
\left(\frac{1}{s}\ln Z^{a,w}(t) + A^{a,w}(t)\right)
>a\epsilon\,\Big|\, X^s(0)=y\right)\nn\\
&&\le
\mathbb{P}^s\left(\sup_{0\le t<\delta}
\frac{1}{s}\ln Z^{a,w}(t)
>{\frac{a\epsilon}{2}}\,\Big|\, X^s(0)=y\right)
+
\mathbb{P}^s\left(\sup_{0\le t<\delta}
A^{a,w}(t)
>{\frac{a\epsilon}{2}}\,\Big|\, X^s(0)=y\right)\nn\\
& &
\end{eqnarray}
By Doob’s inequality and the fact that $Z^{a,w}(t)$ is a mean-one martingale
for every deterministic initial condition $y$, we have
\begin{eqnarray}
&&\mathbb{P}^s\left(\sup_{0\le t<\delta}
\frac{1}{s}\ln Z^{a,w}(t)
>{\frac{a\epsilon}{2}}\,\Big|\, X^s(0)=y\right)\nn\\
&&=
\mathbb{P}^s\left(\sup_{0\le t<\delta}
Z^{a,w}(t)
>e^{s\frac{a\epsilon}{2}}\,\Big|\, X^s(0)=y\right)\nn\\
&&\le
\mathbb{E}^s[Z^{a,w}(\delta)\,\Big|\, X^s(0)=y] e^{-s\frac{a\epsilon}{2}} = e^{-s\frac{a\epsilon}{2}},
\end{eqnarray}
Furthermore, since $X_i^s(t) \in [0,1]$ there exists a constant $C^{a,{w}}$ such  that 
\be
\sup_{0\le t<\delta} A^{a,w}(t) \le C^{a,{w}} \delta
\ee
Thus we conclude
\be
\lim_{\delta\to 0} \limsup_{s\to\infty} \frac{1}{s} \ln \sup_{y\in\Omega^s}\mathbb{P}^s(\sup_{0\le t<\delta}|{X}^{s}(t)\cdot {w} -{X}^s(0)\cdot {w}|>{\epsilon}\,\Big|\, X^s(0)=y) \le -\frac{a\epsilon}{2}
\ee
Since $a$ is arbitrary, this proves \eqref{step1}.

\paragraph{Second step.}
The construction of the compact set ${\cal K}_n$ follows the standard argument in
\cite{kipnis2013scaling}, pag. 272, which we briefly recall for completeness.
From \eqref{step1}, we have obtained an exponential control of the modulus of continuity of the
projections $X^s(\cdot)\cdot w$, uniformly in $w\in\mathbb{R}^N$. More precisely, for every
$\epsilon>0$,
\be
\lim_{\delta\to 0}\limsup_{s\to\infty}
\frac1s \log
\mathbb{P}_p^s\Big(
\sup_{|t-t'|\le\delta}
|X^s(t)\cdot w - X^s(t')\cdot w| > \epsilon
\Big) = -\infty.
\ee
Since this holds for all $w\in\mathbb{R}^N$, it implies an exponential control of the modulus
of continuity of the process $X^s(\cdot)$ itself in $\mathbb{R}^N$.
On the other hand, the state space $\Omega = [0,1]^{\mathbb{T}_N}$ is compact. Therefore,
the compact containment condition is automatically satisfied:
\be
\mathbb{P}_p^s\big( X^s(t)\in \Omega \ \forall t\in[0,T] \big) = 1.
\ee
Combining these two facts, we can apply the standard tightness criterion in the Skorokhod
space $D([0,T],\Omega)$ (see \cite{kipnis2013scaling}), which ensures the existence of a sequence of
compact sets $({\cal K}_n)_{n\in\mathbb{N}}$ such that
\be
\limsup_{s\to\infty}
\frac1s \log \mathbb{P}_p^s(X^s \notin {\cal K}_n) \le -n.
\ee

\end{proof}

\section{Lower bound}
\label{sec-lower}

To prove the lower bound we first provide an alternative representation of the rate function. 
\begin{proposition}
[Rate function via relative entropy]
\label{rate2}
The dynamical part of the rate function \eqref{eq:rate-global} admits the representation 
\be
\label{rate-fct-dem2}
I_{\mathrm{dyn}}(\bm{x})= 
\begin{cases}
\inf_{h\in A({\bm{x}})} 
\left\{\liminf_{s\to\infty}
\frac{1}{s}
\mathscr{H}(\mathbb{P}^{s,h}(\cdot\,|\,X^s(0)=x^s)\,||\, \mathbb{P}^s(\cdot\,|\,X^s(0)=x^s)\right\}
\qquad& \text{ if } A({\bm{x}}) \neq \emptyset\\
+\infty & \text{ otherwise.}
\end{cases}
\ee
where $\mathscr{H}(\cdot||\, \cdot)$ denotes the relative entropy between path measures and 
\be
A({\bm{x}}) = \{h\in C([0,T];\mathbb{R}^{\mathbb{T}_N}) \,|\, \dot{\bm{x}} = F^{h}(\bm{x})\},
\ee
the function $F^{h}$ being defined in \eqref{vector-b}.
The sequence of configurations $(x^s)_{s\in\N}$ is such that $x^s\in\Omega^s$ and $x^s\to x(0)$ as $s\to\infty$, with $x(0)$ the initial configuration of the trajectory $\bm{x} \in D([0,T],\Omega)$.
\end{proposition}
\begin{proof}
Recall the definition of the relative entropy:
\be
\mathscr{H}\big( \mathbb{P}^{s,h}(\cdot \mid X^s(0)=x^s) \,\|\, \mathbb{P}^s(\cdot \mid X^s(0)=x^s) \big)
=
\mathbb{E}^{s,h}\!\left[
\log \frac{d\mathbb{P}^{s,h}}{d\mathbb{P}^s}
\;\middle|\; X^s(0)=x^s
\right]
\ee
where the Radon--Nikodym derivative is understood between the two path
measures conditioned on the same initial configuration.
We compute the Radon--Nikodym derivative using Girsanov formula. Denoting 
\be
c_{i,j}(x)=x_i(1-x_j),\qquad\qquad c^{h}_{i,j}(t,x)=x_i(1-x_j)e^{h_j(t)-h_i(t)},
\ee
we obtain
\be
\log \frac{d\mathbb{P}^{s,h}}{d\mathbb{P}^s}
=
\sum_{i\in\mathbb{T}_N}
\Bigg(
\int_0^T
\log\!\left(\frac{c^h_{i,i+1}(t,X^s(t^-))}{c_{i,i+1}(X^s(t^-))}\right)
\, dJ^s_{i,i+1}(t)
+
\int_0^T s \big( c_{i,i+1}(X^s(t^-)) - c^h_{i,i+1}(t,X^s(t^-)) \big)\,dt
\nn\ee
\be
\qquad\qquad\qquad
+
\int_0^T
\log\!\left(\frac{c^h_{i+1,i}(t,X^s(t^-))}{c_{i+1,i}(X^s(t^-))}\right)
\, dJ^s_{i+1,i}(t)
+
\int_0^T s \big( c_{i+1,i}(X^s(t^-)) - c^h_{i+1,i}(t,X^s(t^-)) \big)\,dt
\Bigg).
\nn
\ee
\be
\ee
Here $X^s(t-)$ denotes the left limit and $J^s_{i,j}(t)$ is the number of jumps from site $i$ to site $j$ up to time $t$. 
Taking expectation under $\mathbb{P}^{s,h}(\cdot \mid X^s(0)=x^s)$ and using 
\be
\mathbb{E}^{s,h}\!\left[ dJ^s_{i,j}(t) \mid \mathcal{F}(t^-) \right]
=
s c^h_{i,j}(t,X^s(t^-))\,dt,
\ee
we get
\begin{eqnarray}
\label{eq:xxx}
\mathscr{H}\big( \mathbb{P}^{s,h}(\cdot \mid X^s(0)=x^s) \,\|\, \mathbb{P}^s(\cdot \mid X^s(0)=x^s) \big)
=
s \mathbb{E}^{s,h}\Bigg(
\int_0^T {j}^h(t,X^s(t))\,dt
\;\Big|\; X^s(0)=x^s
\Bigg)
\end{eqnarray}
where 
\begin{eqnarray}
j^{h}(t,{x}(t)) 
&=&
\sum_{i\in\mathbb{T}_N}\Bigg \{c^{h}_{i,i+1}(t,x(t))\log\Big(\frac{c^{h}_{i,i+1}(t,x(t))}{{c}_{i,i+1}(x(t))}\Big)+   {c}_{i,i+1}(x(t))-{c}^{h}_{i,i+1}(t,x(t))\nn\\
&&
\hspace{1.cm}  + c^{h}_{i+1,i}(t,x(t))\log\Big(\frac{c^{h}_{i+1,i}(t,x(t))}{{c}_{i+1,i}(x(t))}\Big)+ {c}_{i+1,i}(x(t))-{c}^{h}_{i+1,i}(t,x(t))\Bigg\}.
\end{eqnarray}
Let $h \in A(\bm{x})$, i.e. $\dot{\bm{x}}=F^h(\bm{x})$. Since $x^s \to x(0)$, by Theorem \ref{th-lln-gen} we have that
$
\bm{X}^s(\cdot) \to \bm{x}(\cdot)
$
in $\mathbb{P}^{s,h}(\cdot \mid X^s(0)=x^s)$-probability.
Since the function $j^h$ is continuous and bounded on $[0,T]\times\Omega$, it follows that
\be
\lim_{s\to\infty}
\frac{1}{s}
\mathscr{H}\big( \mathbb{P}^{s,h}(\cdot \mid X^s(0)=x^s) \,\|\, \mathbb{P}^s(\cdot \mid X^s(0)=x^s) \big)
=
\int_0^T {j}^h(t,x(t))\,dt
\ee
The expression obtained above is local in time and depends only on the instantaneous configuration.
Therefore, one can identify the
variational representation of the dynamical rate function given in \eqref{rate-fct-dem2}
 using a standard convex duality arguments (see Theorem 5.26 in \cite{shwartz1995large}), i.e.
\be
\label{convex-duality}
\sup_{{h}} {\cal \ell}^{{h}}(t,x(t),\dot{x}(t)) = 
\inf_{{h} \,:\, \dot{x}(t) = F^{{h}}({x}(t))} {j}^{{h}}(t,{x}(t))
\ee
where
\begin{eqnarray}
&&{\cal \ell}^{h}(t,{x}(t),\dot{x}(t)) \nn\\
&&= 
\sum_{i\in\mathbb{T}_N}
\Big[
{{h}_i(t)} \dot{{x}_i}(t) -
{x}_i(t)(1-{x}_{i+1}(t)) (e^{h_{i+1}(t)-h_i(t)}-1) - {x}_{i+1}(t)(1-{x}_i(t)) (e^{h_i(t)-h_{i+1}(t)}-1)\Big]. \nn\\
\end{eqnarray}
Using \eqref{convex-duality} we conclude
\be
I_{\mathrm{dyn}}(\bm{x})= \sup_h \int_0^T  {\cal \ell}^{h}(t,{x}(t),\dot{x}(t)) dt = \inf_{h \,:\, \dot{\bm{x}} = F^{h}(\bm{x})} \int_0^T {j}^{{h}}(t,{x}(t)) dt.
\ee
\end{proof}

We are now ready to conclude
the proof of the large deviation lower bound. 

\begin{proposition}[Lower bound]
\label{lower-prop}
For every open set ${\cal O}\subset D([0,T],\Omega)$ we have
\begin{equation} 
\label{lower}
\liminf_{s\to\infty}\frac{1}{s}\log\mathbb{P}_p^s (\bm{X}^s\in {\cal O}) \ge -\inf_{\bm x\in {\cal O}}I(\bm{x})   
\end{equation}
\end{proposition}
\begin{proof}
Let $B(\bm{x},\delta)$ be the open ball around $\bm{x}$  of radius $\delta$, i.e.
\be
B(\bm{x},\delta) = \{\bm{y}\in D([0,T],\Omega): d({\bm{y}},\bm{x}) < \delta\},
\ee
where $d(\cdot,\cdot)$ is the Skorokhod distance.
If ${\cal O}\subset D([0,T],\Omega)$ is open, then for any $\bm{x} \in {\cal O}$, there is an open ball $B(\bm{x},\delta_{\bm{x}})\subset {\cal O}$. Therefore, for every $\bm{x} \in {\cal O}$
\be
\liminf_{s\to\infty} \frac{1}{s}\ln\mathbb{P}_p^s(\bm{X}^s\in {\cal O}) 
\ge
\liminf_{s\to\infty} \frac{1}{s}\ln\mathbb{P}_p^s(\bm{X}^s\in B(\bm{x},\delta_{\bm{x}})).
\ee
As a consequence, to prove the lower bound \eqref{lower} it is enough to show for a given $\bm{x}
\in D([0,T],\Omega)$ the local lower bound 
\be
\label{needthis}
\liminf_{\delta\to 0}\liminf_{s\to\infty} \frac{1}{s}\ln\mathbb{P}_p^s\big(\bm{X}^s\in B({\bm{x}},\delta)\big) \ge - I(\bm{x}).
\ee
Consider the tilted process $\bm{X}^{s,h}$ with generator \eqref{gen-tilted} initialized from the product measure $\nu_p^s$ and denote by $\mathbb{P}_{x(0)}^{s,h}$ its law in path space. We have
\begin{eqnarray}
\mathbb{P}_p^s(\bm{X}^{s}\in B({\bm{x}},\delta)) 
&=& 
\mathbb{E}_p^s[{\bf 1}_{\{\bm{X}^{s}\in B({\bm{x}},\delta)\}}]
\\
&=& 
\mathbb{E}_{x(0)}^{s,h}[{\bf 1}_{\{\bm{X}^{s,h}\in B({\bm{x}},\delta)\}} \frac{d\mathbb{P}_p^{s}}{d\mathbb{P}_{x(0)}^{s,h}}]\nonumber,
\end{eqnarray}
where $\mathbb{E}_p^{s,h}$ denotes expectation with respect to
$\mathbb{P}_p^{s,h}$.
Therefore, applying Jensen inequality, one gets
\begin{eqnarray}
\label{eq-jensen}
\log(
\mathbb{P}_p^s(\bm{X}^{s}\in B({\bm{x}},\delta)) 
&\ge &
\log(
\mathbb{P}_{x(0)}^{s,h}(\bm{X}^{s,h}\in B({\bm{x}},\delta)) +
\frac{
\mathbb{E}_{x(0)}^{s,h}[{\bf 1}_{\{\bm{X}^{s,h}\in B({\bm{x}},\delta)\}} \ln(\frac{d\mathbb{P}_p^s}{d\mathbb{P}_{x(0)}^{s,h}})]
}
{\mathbb{P}_{x(0)}^{s,h}(\bm{X}^{s,h}\in B({\bm{x}},\delta))}. \nn\\
&&
\end{eqnarray}
Now, for a given path $\bm{x}
\in D([0,T],\Omega)$, choose $h = \{h(t) \,:\, t\ge 0\}$ so that the $\bm{X}^{s,h}$ process 
satisfies a  functional law of large numbers with mean $\bm{x}$. In view of Theorem \ref{th-lln} this amounts to requiring that 
\begin{eqnarray}
\label{lln}
\dot{\bm{x}} = F^{h}(\bm{x})
\end{eqnarray}
where $F^{h}$ is defined in \eqref{vector-b}.
If 
\be
\lim_{s\to\infty} \mathbb{P}_{x(0)}^{s,h}(\bm{X}^{s,h}\in B({\bm{x}},\delta))=1
  \qquad\text{ for all } \delta>0
  \ee
  then, as a consequence of\eqref{eq-jensen},
we get
\begin{eqnarray}
\label{bound-not-optimal}
\liminf_{s\to\infty}
\frac{1}{s}\log\big(
\mathbb{P}_p^s(\bm{X}^{s}\in B({\bm{x}},\delta)\big) 
&\ge& 
\liminf_{s\to\infty}
\frac{1}{s}
\mathbb{E}_{x(0)}^{s,h}\Big[ \ln\big(\frac{d\mathbb{P}_p^s}{d\mathbb{P}_{x(0)}^{s,h}}\big)\Big].
\end{eqnarray}
We now estimate separately the dynamical and the initial contributions.
For the dynamical part, we optimize the bound over $ A(\bm{x})= \{h\in C([0,T];\mathbb{R}^{\mathbb{T}_N}) \,|\, \dot{\bm{x}} = F^{h}(\bm{x})\}$, getting 
\begin{eqnarray}
\label{bound-optimal2}
\liminf_{s\to\infty}
\frac{1}{s}\log\big(
\mathbb{P}_p^s(\bm{X}^{s}\in B({\bm{x}},\delta)\big) 
&\ge& 
\inf_{h\in A({\bm{x}})} \Big\{
\liminf_{s\to\infty}
\frac{1}{s}
\mathbb{E}^{s,h}_{x(0)}\!\left[
\log \frac{d\mathbb{P}^s(\cdot \mid X^s(0))}{d\mathbb{P}^{s,h}(\cdot \mid X^s(0))}
\right]\Big\} \nonumber\\
&+& \liminf_{s\to\infty}
\frac{1}{s}
\mathbb{E}^{s,h}_{x(0)}\!\left[
\log \frac{d\nu^s_p}{d\nu^s_{x(0)}}(X^s(0))
\right].
\end{eqnarray}
Then, using Proposition \ref{rate2}, we obtain
\be
\inf_{h\in A({\bm{x}})} \Big\{
\liminf_{s\to\infty}
\frac{1}{s}
\mathbb{E}^{s,h}_{x(0)}\!\left[
\log \frac{d\mathbb{P}^s(\cdot \mid X^s(0))}{d\mathbb{P}^{s,h}(\cdot \mid X^s(0))}
\right]\Big\}
= - I_{\mathrm{dyn}}(\bm{x}).
\ee
Since the
initial distribution $\nu^s_{x(0)}$ satisfies a large deviation principle with
rate function $I_0$, and since under the tilted measure $\mathbb{P}_{x(0)}^{s,h}$ we have
$X^{s,h}(0) \to x(0)$ in probability, it follows that
\be
\liminf_{s\to\infty}
\frac{1}{s}
\mathbb{E}^{s,h}_{x(0)}\!\left[
\log \frac{d\nu^s_p}{d\nu^s_{x(0)}}(X^s(0))
\right]
= - I_0(x(0)).
\ee
Combining the two contributions, we conclude
\be
\liminf_{s\to\infty}\frac{1}{s}\log \mathbb{P}^s_p(\bm{X}^s \in B(\bm{x},\delta))
\ge
- I_0(x(0)) - I_{\mathrm{dyn}}(\bm{x})
= - I(\bm{x}).
\ee
This proves the result.
\end{proof}

\section{Other models}

The path space LDP formulated above for the partial exclusion process can be proved for other interacting particle systems as well. In particular, the large deviation function of those satisfying a local in time large deviation principle will keep the form of an action functional associated with a Hamiltonian system.

In this section we first consider the case of the ``inclusion process'' \cite{gkr2009} and of the ``harmonic process'' \cite{fgk},
which are both examples of processes with a discrete non-compact state space.
We here highlight the main
differences with respect to the case
of the partial exclusion process
and provide additional steps that
are needed to deal with the unbounded state space.
Then we discuss two models with continuous state space, the Brownian energy process \cite{gkr2009} and an integrable model of heat conduction \cite{fgk}.

Since the sequence of harmonic processes 
(as well as the sequence of integrable heat conduction models) is integrable for all  values of the parameter $s$ (meaning that the corresponding quantum spin chain is integrable for all spin values),
the large deviation function of the
harmonic process is the action associated with an Hamiltonian system that is classically integrable \cite{gs1}.

\subsection{The inclusion process}
\label{sec:inclusion}
Let $\Omega^s=\{0,\frac{1}{s},\frac{2}{s},\ldots\}^{\mathbb{T}_N}$.
The inclusion process
${\bm{X}}^s$ is the Markov process defined by the generator working on functions $f:\Omega^s\to\R$:
\begin{eqnarray}
\label{gen-inclusion}
({\mathscr L}^{s}f)({x}) 
&=& 
s \sum_{i\in {\mathbb{T}_N}} x_i(1+x_{i+1}) (f(x^{i,i+1})-f(x))\nonumber\\
&+& 
s \sum_{i\in {\mathbb{T}_N}} x_{i+1}(1+x_i)(f(x^{i+1,i})-f(x))
\end{eqnarray}
Here, $x^{j,\ell}$ denotes the configuration that is obtained from $x\in\Omega_s$ by moving a mass $\frac{1}{s}$ from site $j$ to site $\ell$, defined in \eqref{newconf}.
At any time $t>0$ the random vector $X^s(t) =(X^s_{1}(t),\ldots, X^s_{N}(t))$, taking values in $\Omega^s$,
describes the rescaled number of particles.
The rates of the inclusion process are very similar to the ones of the partial exclusion process (a minus is changed to a plus). However, in the inclusion process, every site can accumulate an unbounded number of particles, and so the state space is non-compact. 
For a discussion of the inclusion process and its main properties, see \cite{gkr2009}.

We shall use notations similar to those we used for the partial exclusion process. Namely, for a fixed time interval $[0,T]$, we denote by $E^s = D([0,T];\Omega^s)$ the space of cadlag trajectories on $\Omega^s$. Elements of $E^s$ are denoted by ${\bm x}^s=\{x^s(t)\,:\; t\in[0,T]\}$.  Given a vector $\theta=(\theta_1,\ldots,\theta_N)\in(0,\infty)^{\mathbb{T}_N}$, we denote by 
$\mathbb{P}^s_{\theta}$ the law of the process $\bm{X}^s$ initialized from configurations that are distributed according to $\mu^s_\theta$
having mass function
\be
\label{initial-incl-harm}
\mu^s_{\theta}(x) = \bigotimes_{i\in {\mathbb T}_N} \frac{1}{ (s x_i)! } \frac{\Gamma(s+ s x_i )}{\Gamma(s)} \left(\frac{\theta_i}{1+\theta_i}\right)^{s x_i } \left(\frac{1}{1+\theta_i}\right)^{s}, \qquad\qquad x\in\Omega^s.
\ee
The marginal at each site $i$ is a (scaled) negative binomial distribution with mean $\theta_i$.
The homogeneous product measure $\mu_{\bar{\mathscr{\theta}}}$, with $\bar{\theta}$ having all components equal, is reversible.
The measure \eqref{initial-incl-harm}
describes a ``local equilibrium'' initial condition, corresponding to an inhomogenous version of the reversible measure.

The tilted inclusion process  $ \bm{X}^{s,h}$
is the Markov process generated by 
\begin{eqnarray}
\label{gen-inclusion-tilt}
({\mathscr L}^{s,h}f)({x}) 
&=& 
s \sum_{i\in {\mathbb{T}_N}} x_i(1+x_{i+1}) e^{h_{i+1}(t)-h_i(t)}(f(x^{i,i+1})-f(x))\nonumber\\
&+& 
s \sum_{i\in {\mathbb{T}_N}} x_i(1+x_{i+1})  e^{h_{i}(t)-h_{i+1}(t)} (f(x^{i+1,i})-f(x)).
\end{eqnarray}

\subsubsection{Law of large numbers}

\begin{theorem}[LLN, inclusion process]
\label{th-lln-inclusion}
Let $\Omega := (0,\infty)^{\mathbb{T}_N}$.
Consider the tilted inclusion process $\bm{X}^{s,h}$ defined by the generator \eqref{gen-inclusion-tilt}, with initial distributions $(\mu^s_\theta)$. 
Then, for every $T>0$ and every $\epsilon>0$, we have
\begin{equation}
\lim_{s\to\infty}\mathbb P_{\theta}^{s,h}\Big(\sup_{t\in[0,T]}|X^{s,h}(t)-X(t)|>\epsilon\Big)=0,
\label{7.6}
\end{equation}
where $\bm{X}$ is the flow initialized from $X(0)=\theta$ and satisfying the equations
\begin{equation}
\frac{dX(t)}{dt}=F^h(X(t)),
\label{flow-inclusion}
\end{equation}
with $F^h:\Omega\to\mathbb R^{\mathbb T_N}$ given by
\begin{eqnarray}
F^{h}_{i}(x)
&= & 
x_{i-1}(1+x_i)e^{h_i(t)-h_{i-1}(t)}-x_i(1+x_{i-1})e^{h_{i-1}(t)-h_i(t)}\nonumber \\
&- & 
x_i(1+x_{i+1})e^{h_{i+1}(t)-h_i(t)}+x_{i+1}(1+x_i)e^{h_i(t)-h_{i+1}(t)}.
\label{field-flow-inclusion}
\end{eqnarray}
\end{theorem}

\noindent
{\bf Remark.}
The law of large numbers holds for more general initial distributions $(\mu^s)_{s\in\N}$. As in Theorem \ref{th-lln-gen},  for a given $x\in \Omega$, it is enough to assume that $
X^{s,h}(0) \xrightarrow{{\mathbb{P}_{\mu^{s}}}} x $
to have that the tilted inclusion process $X^{s,h}$ started from $\mu^s$ converges in probability to the flow \eqref{flow-inclusion} initialized from $X(0)=x$.
\begin{proof}
The proof follows the same strategy as that of Theorem~\ref{th-lln},
with one important difference: the state space is no longer compact, so $F^h$ is not globally Lipschitz on all of $\Omega$. The conservation of total mass compensates this.
Indeed, for the inclusion process, every jump moves a mass $1/s$ from one site to a neighboring one, and therefore
\begin{equation}
\sum_{i\in\mathbb T_N}X_i^{s,h}(t)=\sum_{i\in\mathbb T_N}X_i^{s,h}(0)
\qquad \forall t\ge 0.
\label{7.9}
\end{equation}
Similarly, summing the components of $F^h$ in \eqref{field-flow-inclusion}, one checks that
$\sum_{i\in\mathbb T_N}F^{h}_{i}(x)=0$,
so that the limiting flow also conserves the total mass:
\begin{equation}
\sum_{i\in\mathbb T_N}X_i(t)=\sum_{i\in\mathbb T_N}\theta_i
\qquad \forall t\ge 0.
\label{7.11}
\end{equation}
Fix $M>\sum_{i\in\mathbb T_N}\theta_i$ and define the event
\begin{equation}
A_M^s=\Bigl\{\sum_{i\in\mathbb T_N}X_i^{s,h}(0)\le M\Bigr\}.
\label{7.12}
\end{equation}
On the event $A_M^s$, by \eqref{7.9} the whole trajectory remains in the set
\be
K_M=\Bigl\{x\in\Omega:\sum_{i\in\mathbb T_N}x_i\le M\Bigr\}.
\ee
On $K_M$, all coordinates are bounded by $M$, and therefore the vector field $F^h$ is Lipschitz on $K_M$. Hence there exists a constant $C_M>0$ such that
\begin{equation}
\label{lipschitz-incl}
|F^h(x)-F^h(y)|\le C_M|x-y|
\qquad \forall x,y\in K_M.
\end{equation}
The martingale decomposition is exactly the same as in Section \ref{section-lln}: if we define $f_i(x)=x_i$, then
\begin{equation}
X_i^{s,h}(t)=X_i^{s,h}(0)+\int_0^t {\mathscr L}^{s,h}f_i(X^{s,h}(u))\,du+M_i^{s,h}(t),
\label{7.15}
\end{equation}
while the limiting flow satisfies
\begin{equation}
X_i(t)=X_i(0)+\int_0^t {\mathscr L}f_i(X(u))\,du.
\label{7.16}
\end{equation}
Since
\be
{\mathscr L}^{s,h}f_i(x)=F^{h}_i(x),\qquad {\mathscr L}f_i(x)=F^{h}_i(x),
\ee
we get
\begin{equation}
X^{s,h}(t)-X(t)=X^{s,h}(0)-X(0)+\int_0^t\bigl(F^h(X^{s,h}(u))-F^h(X(u))\bigr)\,du+M^{s,h}(t).
\label{7.17}
\end{equation}
Using \eqref{lipschitz-incl}, on the event $A_M^s$ we obtain
\begin{equation}
|X^{s,h}(t)-X(t)|
\le
|X^{s,h}(0)-X(0)|+C_M\int_0^t|X^{s,h}(u)-X(u)|\,du+|M^{s,h}(t)|.
\label{7.18}
\end{equation}
Therefore, by the same Gr\"onwall argument as in Theorem~\ref{th-lln},
\begin{align}
&\mathbb P_{\theta}^{s,h}\Bigl(\sup_{t\in[0,T]}|X^{s,h}(t)-X(t)|>\epsilon\Bigr)
\notag\\
&\le
\mathbb P_{\theta}^{s,h}\bigl((A_M^s)^c\bigr)
+\mathbb P_{\theta}^{s,h}\Bigl(|X^{s,h}(0)-X(0)|>\tfrac{\epsilon}{2}e^{-C_MT}\Bigr)
+\mathbb P_{\theta}^{s,h}\Bigl(\sup_{t\in[0,T]}|M^{s,h}(t)|>\tfrac{\epsilon}{2}e^{-C_MT}\,\cap \, A_M^s\Bigr).
\end{align}
Using \eqref{initial-incl-harm} the first term vanishes as $s\to\infty$, and the second one vanishes as well when $X(0) = \theta$. It remains to control the martingale term.
For the quadratic covariation, the only change with respect to \eqref{covar} is that the factors $(1-x_j)$ are replaced by $(1+x_j)$. Hence, on $K_M$, the jump rates are uniformly bounded, and each jump still changes $f_i(x)=x_i$ by at most $1/s$. Therefore there exists a constant $C_M'>0$ such that
\begin{equation}
\Gamma(f_i,f_i)(z)\le \frac{C_M'}{s}
\qquad \forall z\in K_M.
\label{7.20}
\end{equation}
Using Doob's inequality and the isometry for square-integrable martingales, we conclude that for each $i\in\mathbb T_N$,
\begin{align}
&\mathbb P_{\theta}^{s,h}\Bigl(\sup_{t\in[0,T]}|M_i^{s,h}(t)|>\tfrac{\epsilon}{2}e^{-C_MT},\,A_M^s\Bigr)
\notag\\
&\qquad\le
\frac{4e^{2C_MT}}{\epsilon^2}
\,\mathbb E_{\theta}^{s,h}\!\left[\langle M_i^{s,h},M_i^{s,h}\rangle(T)\mathbf{1}_{A_M^s}\right]
\le
\frac{4e^{2C_MT}}{\epsilon^2}\,\frac{C_M'T}{s},
\label{7.21}
\end{align}
which vanishes as $s\to\infty$. This proves \eqref{7.6}.
\end{proof}

\subsubsection{Large deviations}

\begin{theorem}[LDP, inclusion process]
\label{ldp-incl}
The sequence of inclusion processes $(\bm{X}^s)_{s\in\mathbb{N}}$, started from the distribution $(\mu^s_\theta)_{s\in \N}$ in \eqref{initial-incl-harm}, satisfies a large deviation principle on $D([0,T],\Omega)$ with speed $s$ and rate function   
\begin{equation}
I(\bm{x})=I_0(x(0))+I_{\mathrm{dyn}}(\bm{x}),
\label{7.22}
\end{equation}
where the static part is given by
\begin{equation}
I_0(x)
=
\sum_{i\in\mathbb{T}_N}
\left[
x_i \log\frac{x_i}{\theta_i}
-
(1+x_i)\log\frac{1+x_i}{1+\theta_i}
\right]
\end{equation}
and the dynamical part is given by
\begin{equation}
I_{\mathrm{dyn}}(\bm{x})
=
\sup_{h\in C^1([0,T];\mathbb R^{\mathbb T_N})}
\left[
h(T)\cdot x(T)-h(0)\cdot x(0)
-
\int_0^T\bigl(\dot h(t)\cdot x(t)-H(x(t),h(t))\bigr)\,dt
\right],
\label{7.23}
\end{equation}
with Hamiltonian
\begin{equation}
H(x,w)
=
\sum_{i\in\mathbb T_N}
\Bigl(
x_i(1+x_{i+1})(e^{w_{i+1}-w_i}-1)
+
x_{i+1}(1+x_i)(e^{w_i-w_{i+1}}-1)
\Bigr).
\label{7.24}
\end{equation}
\end{theorem}

\begin{proof}
The proof follows the same strategy as that of Theorem \eqref{main-theorem}
. The only substantial difference is that the state space
$
\Omega=[0,\infty)^{\mathbb T_N}
$
is not compact. This is handled by localizing the process on compact subsets determined by the conserved total mass.
For $M>0$, define
\begin{equation}
K'_M=
\Bigl\{
\bm{x}\in D([0,T],\Omega):\sup_{t\in[0,T]}\sum_{i\in\mathbb T_N}x_i(t)\le M
\Bigr\}.
\label{7.25}
\end{equation}
On $K'_M$, all coordinates are uniformly bounded by $M$, so all jump rates of the original and tilted inclusion processes are uniformly bounded. In particular, once restricted to $K'_M$, the arguments used in Sections \ref{sec-upper}  and \ref{sec-lower} for the partial exclusion process apply verbatim.

\smallskip

\noindent
{\bf{Upper bound.}}
Fix a closed set ${\cal C}\subset D([0,T],\Omega)$. For every $M>0$, the set ${\cal C}\cap K'_M$ is closed in the compact state space $\{\sum_i x_i\le M\}$, and hence the proofs in Section \ref{sec-upper} yield
\begin{equation}
\limsup_{s\to\infty}\frac1s\log \mathbb P_{\theta}^s(\bm{X}^s\in {\cal C}\cap K'_M)
\le
-\inf_{\bm{x}\in {\cal C}\cap K'_M} I(\bm{x}).
\label{7.26}
\end{equation}
It remains to control the contribution outside $K'_M$. 
We have
\be
\mathbb P_{\theta}^s(\bm{X}^s\in {\cal C})
\le
\mathbb P_{\theta}^s(\bm{X}^s\in {\cal C}\cap K'_M)
+
\mathbb P_{\theta}^s(\bm{X}^s\notin K'_M).
\ee
By the conservation of the total mass,
\be
\mathbb P_{\theta}^s(\bm{X}^s\notin K'_M)
=
\mathbb{P}_{\mu^s_\theta}\Bigl(\sum_{i\in\mathbb T_N}x_i>M\Bigr).
\ee
Since the initial distributions $(\mu_\theta^s)_{s\in\N}$ satisfy an LDP on $\Omega$ with rate function $I_0$, the large deviation upper bound for the closed set $\{\sum_i x_i\ge M\}$ gives
\begin{equation}
\limsup_{s\to\infty}\frac1s\log \mathbb P_{\theta}^s(\bm{X}^s\notin K'_M)
\le
-\inf_{\sum_i x_i\ge M} I_0(x).
\label{7.27}
\end{equation}
Combining this with \eqref{7.26} and letting $M\to\infty$, we obtain the large deviation upper bound for closed sets.

\smallskip

\noindent
{\bf{Lower bound.}}
Fix an open set ${\cal O}\subset D([0,T],\Omega)$ and a trajectory $\bm{x}\in {\cal O}$ such that $I(\bm{x})<\infty$. Since $\bm{x}$ is càdlàg and has finite total mass on $[0,T]$, there exists $M>0$ such that
\begin{equation}
\sup_{t\in[0,T]}\sum_{i\in\mathbb T_N}x_i(t)<M.
\label{7.28}
\end{equation}
By taking $\delta>0$ small enough, we may assume that
\begin{equation}
B(\bm{x},\delta)\subset {\cal O}\cap K_M.
\label{7.29}
\end{equation}
On the set $K'_M$, the inclusion process has uniformly bounded rates, and therefore  using the tilted inclusion process, Theorem \ref{th-lln-inclusion}, and the relative-entropy representation of the rate function, we obtain
\begin{equation}
\liminf_{s\to\infty}\frac1s\log \mathbb P_{\theta}^s(\bm{X}^s\in B(\bm{x},\delta))
\ge
-I(\bm{x}).
\label{7.30}
\end{equation}
Since $B(\bm{x},\delta)\subset {\cal O}$, this yields
\be
\liminf_{s\to\infty}\frac1s\log \mathbb P_{\theta}^s(\bm{X}^s\in {\cal O})
\ge
-I(\bm{x}).
\ee
Taking the infimum over $\bm{x}\in {\cal O}$ proves the lower bound for open sets.
\end{proof}

\subsection{The harmonic process}
\label{sec:harmonic}
Let $\Omega^s=\{0,\frac{1}{s},\frac{2}{s},\ldots\}^{\mathbb{T}_N}$.
The harmonic process  $ \bm{X}^s$, taking values in $\Omega_s$,
is the Markov process defined by the generator:
\begin{eqnarray}
\label{gen-harm}
({\cal L}^{s}f)({x}) 
&=& 
s \sum_{i\in {\mathbb{T}_N}} \sum_{k=1}^{s x_i }
\varphi_s(k, sx_i) (f(x^{i,i+1,k})-f(x))\nonumber\\
&+& 
s \sum_{i\in {\mathbb{T}_N}} \sum_{k=1}^{ s x_{i+i}}\varphi_s(k,sx_{i+1}) (f(x^{i+1,i,k})-f(x)), \qquad\qquad x\in\Omega^s,
\end{eqnarray}
where  the rates are given by
\be
\label{rates-harm}
\varphi_s(k,n) = \frac{1}{k} \frac{\Gamma(n+1)}{\Gamma(n+s)} \frac{\Gamma(n-k+s)}{\Gamma(n-k+1)} \indic{k\le n} \qquad\qquad k,n\in\mathbb{N}.
\ee
Here, $x^{j,\ell, k}$ denotes the configuration that is obtained from $x\in\Omega_s$ by moving a mass $\frac{k}{s}$ from site $j$ to site $\ell$, i.e.
\begin{eqnarray}
\label{newconf-harmonic}
x^{j,\ell,k}_i =
\begin{cases}
x_i, &  i\neq j,\ell,\\
x_i-\frac{k}{s},\qquad &  i = j,\\
x_i+\frac{k}{s}, & i = \ell.
\end{cases}
\end{eqnarray}
As in the zero-range process \cite{spitzer1991interaction}, in the harmonic process the rate for a movemenent of particles from a site $j$ to a site $\ell$ is a function of the number of particles at $j$ (the departure site) only, it does not depend on the number of particle at $\ell$ (the arrival site). However, in the harmonic process a jump of many particles is possible, while in the zero-range process only jumps of one particle are allowed. For a discussion of the harmonic process and its main properties, see \cite{fgk}.

Similarly to the inclusion process, the harmonic process has a family of reversible measures given a homogeneous product of negative binomial distributions (the label of the family being the mean).  
Therefore, we shall consider the harmonic process started from the initial measure $\mu^s_{\theta}$ defined in \eqref{initial-incl-harm}, corresponding to a local equilibrium condition.

For the proof of the large deviations lower bound, we will need a tilted harmonic process. This is the Markov process $ \bm{X}^{s,h}$ generated by 
\begin{eqnarray}
\label{gen-harm-tilt}
({\cal L}^{s,h}f)({x}) 
&=& 
s \sum_{k=1}^{s x_i}
\varphi_s(k,sx_i) e^{k(h_{i+1}(t)-h_i(t))} (f(x^{i,i+1,k}-f(x))\nonumber\\
&+& 
s \sum_{k=1}^{s x_{i+1}}\varphi_s(k,sx_{i+1}) e^{k(h_{i}(t)-h_{i+1}(t))} (f(x^{i+1,i,k})-f(x)),\qquad\qquad x\in\Omega^s.
\end{eqnarray}
As in the previous cases (partial exclusion process, inclusion process) the tilting here is the exponential of the discrete gradient of the continuous function $h:\mathbb{R} \to\mathbb{R}^{\mathbb{T}_N}$.
However, the tilting is now also exponential in $k$, the number of particles moved in the transition.

\noindent
{\bf Remark.}
Before treating the harmonic process, we point out an important difference from
the partial exclusion and inclusion processes. In both of the latter models, the
only possible jumps are nearest-neighbour transfers of one unit of mass \(1/s\).
Thus the macroscopic jump vectors belong to a fixed finite set.
For the harmonic process this is no longer true. A single jump may transfer
\(k/s\) units of mass from one site to a neighbouring site, with
\(k=1,\ldots,sx_i\). Hence the microscopic generator contains a sum over the
jump size \(k\), and the corresponding macroscopic jump vectors are
\[
    k(e_{i+1}-e_i), \qquad k(e_i-e_{i+1}), \qquad k\geq 1 .
\]
Thus the model cannot be represented by a fixed finite set of jump channels as in \eqref{network-generator}.
This distinction is also reflected in the Hamiltonian. Finite-jump models give
Hamiltonians that are finite sums of exponentials. In the harmonic case the sum
over \(k\) produces a logarithm.
Consequently the limiting Hamiltonian has a natural finite domain. This domain
constraint plays an essential role in the tilted dynamics and in the
large-deviation lower bound.

\subsubsection{Law of large numbers}

The following law of large numbers is again of Kurtz type. One additional
point, compared with the partial exclusion and inclusion processes, is that the
harmonic process allows jumps of size \(k/s\), with \(k\geq 1\). This motivates the following definition.
\begin{definition}
[Admissible pairs]
\label{def-admissible}
Let $\Omega := (0,\infty)^{\mathbb{T}_N}$ and let $h \in C^1([0,T];\mathbb{R}^{\mathbb{T}_N})$. We say that a pair $(\bm{x},h)$,
with $\bm{x} \in C([0,T];\Omega)$, is admissible on $[0,T]$ if
\begin{enumerate}
\item $\bm{x}$ solves the Cauchy problem
\begin{equation}
\dot{x}(t)=F^h(x(t)), \qquad x(0)=\theta,
\label{7.39}
\end{equation}
where $F^h:\Omega\to\mathbb R^{\mathbb T_N}$ is given by
\begin{eqnarray}
\label{drift-harm}
F^h_i(x)&=&
\left(
\frac{x_{i-1}e^{h_i(t)-h_{i-1}(t)}}{1-x_{i-1}(e^{h_i(t)-h_{i-1}(t)}-1)}
-
\frac{x_ie^{h_{i-1}(t)-h_i(t)}}{1-x_i(e^{h_{i-1}(t)-h_i(t)}-1)}
\right)\nonumber\\
&-&
\left(
\frac{x_ie^{h_{i+1}(t)-h_i(t)}}{1-x_i(e^{h_{i+1}(t)-h_i(t)}-1)}
-
\frac{x_{i+1}e^{h_i(t)-h_{i+1}(t)}}{1-x_{i+1}(e^{h_i(t)-h_{i+1}(t)}-1)}
\right);
\label{7.42}
\end{eqnarray}
\item along the trajectory one has
\be
\inf_{t\in[0,T]}\min_{i\in \mathbb{T}_N}
\left\{
1 - x_i(t)\bigl(e^{h_{i+1}(t)-h_i(t)} - 1\bigr),
\;
1 - x_{i+1}(t)\bigl(e^{h_i(t)-h_{i+1}(t)} - 1\bigr)
\right\} > 0.
\ee
\end{enumerate}
\end{definition}

\begin{theorem}[LLN, harmonic process]
\label{lln-harmonic}
Consider the tilted harmonic process $\bm{X}^{s,h}$ defined by the generator \eqref{gen-harm}, with initial distribution $\mu^s_\theta$.
Assume that the pair $(\bm{X},h)$ is admissible on $[0,T]$, where \(X\) solves
\begin{equation}
\frac{dX(t)}{dt}=F^h(X(t)), \qquad X(0)=\theta.
\label{flow-harm}
\end{equation}
Then, for every $\epsilon>0$,
\begin{equation}
\lim_{s\to\infty}
\mathbb P_{\theta}^{s,h}
\Bigl(
\sup_{t\in[0,T]}|X^{s,h}(t)-X(t)|>\epsilon
\Bigr)=0.
\label{7.41}
\end{equation}
\end{theorem}

\begin{proof}
The proof follows the same strategy as those of Theorems \ref{th-lln} and \ref{th-lln-inclusion}. We only discuss the additional points.
As in the inclusion process, the total mass is conserved. Indeed, every jump of the harmonic process moves a mass $k/s$ from one site to a neighboring one, and therefore
\begin{equation}
\sum_{i\in\mathbb T_N}X_i^{s,h}(t)=\sum_{i\in\mathbb T_N}X_i^{s,h}(0)
\qquad \forall t\ge 0.
\label{7.43}
\end{equation}
Similarly, summing the components of the vector field in \eqref{7.42}, one checks that
\begin{equation}
\sum_{i\in\mathbb T_N}F_i^h(x)=0,
\label{7.44}
\end{equation}
so that the limiting flow also conserves the total mass. Fix $M>\sum_{i\in\mathbb T_N}\theta_i$ and define
\begin{equation}
A_M^s=
\left\{
\sum_{i\in\mathbb T_N}X_i^{s,h}(0)\le M
\right\}.
\label{7.45}
\end{equation}
On the event $A_M^s$, by \eqref{7.43} the whole trajectory remains in the set
\be
K_M=\left\{x\in\Omega:\sum_{i\in\mathbb T_N}x_i\le M\right\}.
\ee
By admissibility and continuity, there exists $\delta>0$ such that the denominators in \eqref{7.42} remain strictly positive on the $\delta$-neighborhood of the trajectory $X(\cdot)$. In particular, on this neighborhood the vector field $F^h$ is Lipschitz. Hence there exists a constant $C>0$ such that
\begin{equation}
|F^h(x)-F^h(y)|\le C|x-y|
\label{7.47}
\end{equation}
whenever $x$ and $y$ belong to that neighborhood.
As in Section \ref{section-lln}, let $f_i(z)=z_i$. Then
\begin{equation}
X_i^{s,h}(t)
=
X_i^{s,h}(0)
+
\int_0^t {\cal L}^{s,h}f_i(X^{s,h}(u))\,du
+
M_i^{s,h}(t),
\label{7.48}
\end{equation}
where $\bm{M}^{s,h}$ is a martingale vector. The difference with the partial exclusion and inclusion cases is that now
$
{\cal L}^{s,h}f_i(z)\neq F_i^h(z)
$
for finite $s$. 
A direct computation gives
\be
\begin{aligned}
{\cal L}^{s,h}f_i(z)
&=
\sum_{k\ge 1} k\,\varphi_s(k,sz_{i-1})e^{k(h_i(t)-h_{i-1}(t))}
-\sum_{k\ge 1} k\,\varphi_s(k,sz_i)e^{k(h_{i-1}(t)-h_i(t))}
\\
&\quad
-\sum_{k\ge 1} k\,\varphi_s(k,sz_i)e^{k(h_{i+1}(t)-h_i(t))}
+\sum_{k\ge 1} k\,\varphi_s(k,sz_{i+1})e^{k(h_i(t)-h_{i+1}(t))}.
\end{aligned}
\ee
Here the sums are finite, since $\varphi_s(k,n)=0$ for $k>n$.
However, we can show that, uniformly on compact subsets where the denominators in \eqref{drift-harm} are bounded away from zero,
\begin{equation}
\sup_{z}\bigl|{\cal L}^{s,h}f_i(z)-F_i^h(z)\bigr|\xrightarrow[s\to\infty]{}0.
\label{7.49}
\end{equation}
Indeed, using
\be
k\varphi_s(k,sz)
=
\prod_{\ell=0}^{k-1}
\frac{sz-\ell}{s(1+z)-1-\ell},
\ee
we have, for fixed $k$,
\be
\label{asymptotic1}
k\varphi_s(k,sz)
\xrightarrow[s\to\infty]{}
\left(\frac{z}{1+z}\right)^k .
\ee
Hence, for $a\in\mathbb R$,
\be
\sum_{k\ge1} k\varphi_s(k,sz)e^{ka}
\xrightarrow[s\to\infty]{}
\sum_{k\ge1}
\left(\frac{ze^a}{1+z}\right)^k
=
\frac{ze^a}{1-z(e^a-1)}
\ee
and the convergence is uniform on compact sets on which $1-z(e^a-1)$
is bounded away from zero, because the above series are then dominated by a
geometric series uniformly in $s$. Applying this estimate we get \eqref{7.49}.
Hence, writing
\begin{equation}
r_i^s(z)={\cal L}^{s,h}f_i(z)-F_i^h(z),
\label{7.50}
\end{equation}
we have
\begin{equation}
X^{s,h}(t)-X(t)
=
X^{s,h}(0)-X(0)
+
\int_0^t\bigl(F^h(X^{s,h}(u))-F^h(X(u))\bigr)\,du
+
\int_0^t r^s(X^{s,h}(u))\,du
+
M^{s,h}(t).
\label{7.51}
\end{equation}
Therefore, on the event that $X^{s,h}(\cdot)$ remains in the $\delta$-neighborhood of $X(\cdot)$, we obtain
\begin{equation}
|X^{s,h}(t)-X(t)|
\le
|X^{s,h}(0)-X(0)|
+
C\int_0^t|X^{s,h}(u)-X(u)|\,du
+
T\sup|r^s|
+
|M^{s,h}(t)|.
\label{7.52}
\end{equation}
Applying Gr\"onwall's inequality exactly as in Theorems  \ref{th-lln} and \ref{th-lln-inclusion}, we are reduced to controlling the initial condition, the error term $r^s$, and the martingale term.
The initial term vanishes when $X(0)=\theta$, while the error term vanishes by \eqref{7.49}. It remains to estimate the martingale term. In contrast with the inclusion process, jumps of size $k/s$ with $k\ge 1$ are possible. Nevertheless, on every compact set where the denominators are bounded away from zero, the tilted rates have uniformly bounded second moment, and therefore for each $i\in\mathbb T_N$ there exists a constant $C'_M>0$ such that
\begin{equation}
\Gamma(f_i,f_i)(z)\le \frac{C'_M}{s}
\qquad \forall z\in K_M.
\label{7.53}
\end{equation}
Using Doob's inequality and the isometry for square-integrable martingales, we conclude that
\begin{equation}
\mathbb P_{\nu^s}^{s,h}
\Bigl(
\sup_{t\in[0,T]}|M_i^{s,h}(t)|>\epsilon,\,A_M^s
\Bigr)
\le
\frac{4}{\epsilon^2}
\mathbb E_{\nu^s}^{s,h}
\bigl[
\langle M_i^{s,h},M_i^{s,h}\rangle(T)\mathbf 1_{A_M^s}
\bigr]
\le
\frac{4C'_M T}{\epsilon^2 s},
\label{7.54}
\end{equation}
which vanishes as $s\to\infty$. This proves \eqref{7.41}.
\end{proof}

\subsubsection{Large deviations}

\begin{theorem}[LDP, harmonic process]
\label{main-theorem-harmonic}
The sequence of harmonic processes $(\bm{X}^s)_{s\in\mathbb N}$, 
 initialized from the distribution $(\mu^s_{\theta})_{s\in\mathbb N}$ in \eqref{initial-incl-harm},  satisfies a large deviation principle on $D([0,T],\Omega)$ with speed $s$ and rate function
\begin{equation}
I(\bm{x})=I_0(x(0))+I_{\mathrm{dyn}}(\bm{x}),
\label{7.55}
\end{equation}
where the initial cost is given by
\begin{equation}
\label{static-I0-harm}
I_0(x)
=
\sum_{i\in\mathbb{T}_N}
\left[
x_i \log\frac{x_i}{\theta_i}
-
(1+x_i)\log\frac{1+x_i}{1+\theta_i}
\right].
\end{equation}
and the dynamical part is given by
\begin{equation}
I_{\mathrm{dyn}}(\bm{x})
=
\sup_{h\in C^1([0,T];\mathbb{R}^{T_N})}
\left[
h(T)\cdot x(T)-h(0)\cdot x(0)
-
\int_0^T\bigl(\dot h(t)\cdot x(t)-H(x(t),h(t))\bigr)\,dt
\right],
\label{I-dyn-harm}
\end{equation}
with Hamiltonian
\begin{equation}
H(x,w)
=
-\sum_{i\in\mathbb T_N}
\log\Bigl(1-x_i(e^{w_{i+1}-w_i}-1)\Bigr)
-\sum_{i\in\mathbb T_N}
\log\Bigl(1-x_{i+1}(e^{w_i-w_{i+1}}-1)\Bigr),
\label{ham-harm}
\end{equation}
with the convention that $H(x,w)=+\infty$ whenever one of the logarithms
is not well defined.
\end{theorem}
\begin{proof}
The proof follows the same strategy as that of Theorem \ref{main-theorem}. The only substantial differences are: (a) the state space $\Omega=[0,\infty)^{\mathbb T_N}$ is not compact; (b) the harmonic Hamiltonian \eqref{ham-harm} is finite only on its natural domain.
Both issues are handled by localization: 
for $M>0$, define
\begin{equation}
K'_M=
\left\{
\bm{x}\in D([0,T],\Omega):
\sup_{t\in[0,T]}\sum_{i\in\mathbb T_N}x_i(t)\le M
\right\}.
\label{7.58}
\end{equation}
Since every jump of the harmonic process moves a mass from one site to a neighboring one, the total mass is conserved:
\begin{equation}
\sum_{i\in\mathbb T_N}X_i^s(t)=\sum_{i\in\mathbb T_N}X_i^s(0)
\qquad \forall t\ge 0.
\label{7.59}
\end{equation}
Hence every trajectory starting from total mass at most $M$ remains in $K'_M$.
On this set $K'_M$, all coordinates are uniformly bounded.  In particular, the microscopic drift converges uniformly to the limiting drift, and the jump kernel has uniformly bounded second moment. 

\medskip

\noindent
\textbf{Upper bound.}
The exponential martingale associated with
the harmonic process reads
\be
\label{Z_T-harm}
Z^{s,h}(T)=\exp\{s I_{\rm dyn}^{s,h}(\bm{X}^s)\},
\ee
where
\be
I_{\rm dyn}^{s,h}(\bm{x})
=
h(T)\cdot {x}(T)-h(0)\cdot {x}(0)
-\int_0^T
\left[
\dot h(u)\cdot {x}(u)-H^s({x}(u),h(u))
\right]du
\ee
and where $H^s(x,w)$ is defined, for $x\in\Omega^s$ and $w\in \R^N$, as
\be
H^s(x,w)
=
\sum_{i \in \mathbb{T}_N}
\Big(\sum_{k=1}^{sx_i}\varphi_s(k,sx_i)
(e^{k(w_{i+1}-w_i)}-1)
+
\sum_{k=1}^{sx_{i+1}}\varphi_s(k,sx_{i+1})(e^{k(w_i-w_{i+1})}-1)
\Big).
\ee
We observed in \eqref{asymptotic1} that
\be
\label{asymptotic}
\varphi_s(k,sx_i) = \frac{1}{k}\left(\frac{x_i}{1+x_i}\right)^k (1+o(1)), \qquad\qquad \text{as } s\to\infty.
\ee
Using the identity
\be
-\log(1-z) = \sum_{k=1}^\infty \frac{z^k}{k}, \qquad\qquad |z| <1,
 \ee
it follows that
\be
H^s(x,w) \longrightarrow H(x,w)\qquad\qquad \text{as } s\to\infty,
\ee
uniformly on compact sets where $H(x,w)$ is finite.
In particular, for each fixed $h$, it holds uniformly on $K'_M$ provided the logarithmic arguments in the definition of $H$ given in \eqref{ham-harm} are bounded away from zero along $K'_M$.

Let us consider an open set
${\cal O}\subset D([0,T],\Omega)$. We  fix $a\in \mathbb R^{\mathbb T_N}$ and write
\be
\mathbb{P}_{\theta}^s(\bm{X}^s\in {\cal O})
=
\mathbb{E}_{\theta}^s\left[
\mathbf 1_{\{\bm{X}^s\in {\cal O}\}}
Z^{s,h}(T)e^{s a\cdot {X}^s(0)}
\cdot \frac{e^{-s a\cdot {X}^s(0)}}{Z^{s,h}(T)}
\right].
\ee
Using \eqref{Z_T-harm}, we obtain
\be
\mathbb{P}_{\theta}^s(\bm{X}^s\in {\cal O})
\le
\mathbb{E}_{\theta}^s
\left[
Z^{s,h}(T)e^{s a\cdot {X}^s(0)}
\right]
\exp\left\{
s\sup_{\bm{x}\in O}
\left[
-I_{\rm dyn}^{s,h}(\bm{x})-a\cdot {x}(0)
\right]
\right\}.
\ee
Since $\bm{Z}^{s,h}$ is a mean-one martingale conditionally on the initial
configuration,
\be
\mathbb{E}_{\theta}^s
\left[
Z^{s,h}(T)e^{s a\cdot {X}^s(0)}
\right]
=
\mathbb{E}_{\mu_{\theta}^s}
\left[
e^{s a\cdot {X}^s(0)}
\right].
\ee
Under $\mu_{\theta}^s$, the variables $s{X}_i^s(0)$ are independent negative binomial random
variables with parameters $(s,\theta_i)$. Hence
\be
\mathbb{E}_{\mu_{\theta}^s}
\left[
e^{s a\cdot {X}^s(0)}
\right]
=
\prod_{i\in\mathbb T_N}
\left(1+\theta_i-\theta_i e^{a_i}\right)^{-s}
=
\exp\left\{
-s\sum_{i\in\mathbb T_N}
\log(1+\theta_i-\theta_i e^{a_i})
\right\}.
\ee
Here and below the expressions involving $a$ are understood only on the domain
$1+\theta_i-\theta_i e^{a_i}>0$ for all $i\in T_N$.
Therefore
\be
\frac1s\log \mathbb{P}_\theta^s(\bm{X}^s\in {\cal O})
\le
\sup_{\bm{x}\in {\cal O}}
\left[
-I_{\rm dyn}^{s,h}(\bm{x})-a\cdot {x}(0)
-\sum_{i\in\mathbb T_N}\log(1+\theta_i-\theta_i e^{a_i})
\right].
\ee
Since this bound holds for every $h$ and every $a$, and using the uniform convergence $H^s\to H$ on compact sets, we obtain
\be
\limsup_{s\to\infty}
\frac1s\log \mathbb{P}_\theta^s(\bm{X}^s\in {\cal O})
\le
\inf_{\substack{h\in C^1([0,T];\mathbb R^{\mathbb T_N})\\
a\in\mathbb R^{\mathbb T_N}}}
\sup_{\bm{x}\in O}
\left[
-I_{\rm dyn}^h(\bm{x})-a\cdot {x}(0)
-\sum_{i\in\mathbb T_N}\log(1+\theta_i-\theta_i e^{a_i})
\right].
\ee
where
\be
I^h_{\mathrm{dyn}}(\bm{x})
=
h(T)\cdot x(T)-h(0)\cdot x(0)
-
\int_0^T\bigl(\dot h(t)\cdot x(t)-H(x(t),h(t))\bigr)\,dt,
\end{equation}
with $H(x,w)$ defined in \eqref{ham-harm}.
Using Lemma 3.3 in Appendix 2 of \cite{kipnis2013scaling}, for every compact
set ${\cal K}\subset D([0,T],\Omega)$ we get
\be
\begin{aligned}
\limsup_{s\to\infty}
\frac1s\log \mathbb{P}_\theta^s(\bm{X}^s\in {\cal K})
&\le
-\inf_{\bm{x}\in {\cal K}}
\sup_{\substack{h\in C^1([0,T];\mathbb R^{\mathbb T_N})\\
a\in\mathbb R^{\mathbb T_N}}}
\left[
I_{\rm dyn}^h(\bm{x})+a\cdot {x}(0)
+\sum_{i\in\mathbb T_N}\log(1+\theta_i-\theta_i e^{a_i})
\right].
\end{aligned}
\ee
The supremum separates into the dynamical and static parts:
\be
\sup_{h} I_{\rm dyn}^h(\bm{x})=I_{\rm dyn}(\bm{x}),
\ee
and
\be
\sup_{a\in\mathbb R^{\mathbb T_N}}
\left\{
a\cdot {x}(0)
+\sum_{i\in\mathbb T_N}\log(1+\theta_i-\theta_i e^{a_i})
\right\}
=
I_0({x}(0)),
\ee
where $I_0$ is defined in \eqref{static-I0-harm}.
Consequently,
\be
\limsup_{s\to\infty}
\frac1s\log \mathbb{P}_\theta^s(\bm{X}^s\in {\cal K})
\le
-\inf_{\bm{x}\in {\cal K}}
\left[
I_{\rm dyn}(\bm{x})+I_0({x}(0))
\right] = -\inf_{\bm{x}\in {\cal K}}I(\bm{x}).
\ee

To extend the upper bound to closed sets, we use localization. Fix  $M>0$ and let ${\cal C}\subset D([0,T],\Omega)$ be closed. Then
\be
\mathbb{P}^s_\theta(\bm{X}^s\in {\cal C}) \le \mathbb{P}^s_\theta(\bm{X}^s\in {\cal C}\cap K'_M) + \mathbb{P}^s_\theta(\bm{X}^s\notin K'_M).
\ee
We claim that the first term satisfies the desired upper bound. Indeed, on the set $K'_M$ all coordinates are uniformly bounded, and therefore the jump rates of the harmonic process are uniformly bounded. In particular, the exponential tightness argument used in Proposition
\ref{prop-tightness} applies verbatim on $K'_M$. Hence, the large deviation upper bound for compact sets extends to closed subsets of $K'_M$, and we obtain
\be
\limsup_{s\to\infty} \frac{1}{s}\log \mathbb{P}^s_\theta(\bm{X}^s \in {\cal C} \cap K'_M)
\le
-\inf_{\bm{x}\in C \cap K'_M} I(\bm{x}).
\ee
For the second term, using again conservation of mass,
\be
\mathbb{P}^s_\theta(\bm{X}^s\notin K'_M) = \mathbb{P}_{\mu^s_\theta}\left( \sum_{i\in T_N} X^s_i(0) > M \right),
\ee
and since the initial distributions satisfy an LDP with rate function $I_0$, we have
\be
\limsup_{s\to\infty} \frac{1}{s}\log \mathbb{P}^s_\theta(\bm{X}^s\notin K_M)
\le 
-\inf_{\substack{x \in \Omega \\ \sum_{i\in T_N} x_i \ge M}} I_0(x).
\ee
Letting $M\to\infty$, the right-hand side diverges to $-\infty$, and we conclude the upper bound for closed sets.

\medskip

\paragraph{Lower bound.}
We prove the lower bound by adapting the argument of Section \ref{sec-lower}. Let
${\cal O}\subset D([0,T],\Omega)$ be open and let $\bm{x}\in {\cal O}$ be such that
$I(\bm{x})<\infty$. It is enough to prove the local lower bound
\be
\liminf_{\delta\downarrow 0}\liminf_{s\to\infty}
\frac1s \log \mathbb{P}^s_\theta\bigl(\bm{X}^s\in B(\bm{x},\delta)\bigr)
\ge -I(\bm{x}),
\ee
where $B(\bm{x},\delta)$ denotes the open ball around $\bm{x}$ in the Skorokhod
topology.
If $I=\infty$ then there is nothing to prove.
Thus we may assume that
$I(\bm{x})<\infty$ and then
necessarily (proof by contradiction) there exists a tilt
$h\in C^1([0,T];\mathbb R^{T_N})$ such that
\be
\dot x(t)=F^h(x(t)),
\ee
and such that 
the couple $(\bm{x},h)$
is admissible in the sense
of Definition \ref{def-admissible}.
Since $\bm{x}$ has finite total mass on $[0, T]$, there exists $M > 0$ such that
\be
\sup_{t \in [0, T]} \sum_{i \in T_N} x_i(t) < M.
\ee
By continuity of the Skorokhod topology, we may choose $\delta > 0$ small enough such that
\be
B(\bm{x}, \delta) \subset {\cal O} \cap K'_M,
\ee
and such that the logarithmic denominators remain bounded away from zero on $B(\bm{x}, \delta)$. 
For such paths, the assumptions of Theorem \ref{lln-harmonic} are satisfied by construction. Therefore,
denoting by $\mathbb{P}^{s,h}_{x(0)}$ the law of the tilted harmonic process with
generator \eqref{gen-harm-tilt}, initialized from the product measure
$\mu^s_{x(0)}$, we have
\be
\mathbb{P}^{s,h}_{x(0)}\bigl(\bm{X}^s\in B(\bm{x},\delta)\bigr)\longrightarrow 1, \qquad\qquad \text{as } s\to\infty.
\ee
Using the Radon--Nikodym derivative between $\mathbb{P}^s_\theta$ and
$\mathbb{P}^{s,h}_{x(0)}$, we write
\be
\mathbb{P}^s_\theta(\bm{X}^s\in B(\bm{x},\delta))
=
\mathbb{E}_{\mathbb{P}^{s,h}_{x(0)}}\left[
\mathbf 1_{\{\bm{X}^s\in B(\bm{x},\delta)\}}
\frac{d\mathbb{P}^s_\theta}{d\mathbb{P}^{s,h}_{x(0)}}(\bm{X}^s)
\right].
\ee
By Jensen's inequality,
\be
\begin{aligned}
\frac1s \log \mathbb{P}^s_\theta(\bm{X}^s\in B(\bm{x},\delta))
&\ge
\frac1s \log \mathbb{P}^{s,h}_{x(0)}(\bm{X}^s\in B(\bm{x},\delta))
\\
&\quad
+
\frac{1}{s\,\mathbb{P}^{s,h}_{x(0)}(\bm{X}^s\in B(\bm{x},\delta))}
\mathbb{E}_{\mathbb{P}^{s,h}_{x(0)}}\left[
\mathbf 1_{\{\bm{X}^s\in B(\bm{x},\delta)\}}
\log \frac{d\mathbb{P}^s_\theta}{d\mathbb{P}^{s,h}_{x(0)}}(\bm{X}^s)
\right].
\end{aligned}
\ee
The first term vanishes in the limit $s\to\infty$. We now compute the
second term. The Radon--Nikodym derivative separates into an initial part
and a dynamical part:
\be
\log \frac{d\mathbb{P}^s_\theta}{d\mathbb{P}^{s,h}_{x(0)}}(\bm{X}^s)
=
\log \frac{d\mu^s_\theta}{d\mu^s_{x(0)}}(X^s(0))
+
\log \frac{d\mathbb{P}^s(\,\cdot\,|X^s(0))}
{d\mathbb{P}^{s,h}(\,\cdot\,|X^s(0))}({\bm X}^s).
\ee
For the initial contribution, since the product measures
$\mu^s_\theta$ satisfy an LDP with rate function $I_0$, and since
under $\mathbb{P}^{s,h}_{x(0)}$ one has $X^s(0)\to x(0)$ in probability,
\be
\lim_{s\to\infty}
\frac1s
\mathbb{E}_{\mathbb{P}^{s,h}_{x(0)}}\left[
\log \frac{d\mu^s_\theta}{d\mu^s_{x(0)}}(X^s(0))
\right]
=
-I_0(x(0)).
\ee
We turn to the dynamical contribution. For $k\ge1$, denote by
$J^s_{i,i+1,k}(t)$ the number of jumps up to time $t$ in which a mass
$k/s$ is moved from $i$ to $i+1$, and define $J^s_{i+1,i,k}(t)$
analogously. The original rates are
\be
c^s_{i,i+1,k}(z)=\varphi_s(k,sz_i),\qquad
c^s_{i+1,i,k}(z)=\varphi_s(k,sz_{i+1}),
\ee
whereas the tilted rates are
\be
c^{s,h}_{i,i+1,k}(t,z)
=
\varphi_s(k,sz_i)e^{k(h_{i+1}(t)-h_i(t))},
\ee
and
\be
c^{s,h}_{i+1,i,k}(t,z)
=
\varphi_s(k,sz_{i+1})e^{k(h_i(t)-h_{i+1}(t))}.
\ee
Girsanov's formula for jump processes gives
\be
\begin{aligned}
\log \frac{d\mathbb{P}^s(\,\cdot\,|X^s(0))}
{d\mathbb{P}^{s,h}(\,\cdot\,|X^s(0))}
&=
-\sum_{i\in \mathbb{T}_N}\sum_{k\ge1}
\int_0^T
k(h_{i+1}(t)-h_i(t))\,dJ^s_{i,i+1,k}(t)
\\
&\quad
-\sum_{i\in \mathbb{T}_N}\sum_{k\ge1}
\int_0^T
k(h_i(t)-h_{i+1}(t))\,dJ^s_{i+1,i,k}(t)
\\
&\quad
+s\sum_{i\in \mathbb{T}_N}\sum_{k\ge1}
\int_0^T
\bigl[
c^{s,h}_{i,i+1,k}(t,X^s(t))
-c^s_{i,i+1,k}(X^s(t))
\bigr]\,dt
\\
&\quad
+s\sum_{i\in \mathbb{T}_N}\sum_{k\ge1}
\int_0^T
\bigl[
c^{s,h}_{i+1,i,k}(t,X^s(t))
-c^s_{i+1,i,k}(X^s(t))
\bigr]\,dt .
\end{aligned}
\ee
Taking expectation under $\mathbb{P}^{s,h}_{x(0)}$, using the compensators of
the jump processes, and dividing by $s$, we obtain
\be
\frac1s
\mathbb{E}_{\mathbb{P}^{s,h}_{x(0)}}\left[
\log \frac{d\mathbb{P}^s(\,\cdot\,|X^s(0))}
{d\mathbb{P}^{s,h}(\,\cdot\,|X^s(0))}
\right]
=
-
\mathbb{E}_{\mathbb{P}^{s,h}_{x(0)}}\left[
\int_0^T j^{s,h}(t,X^s(t))\,dt
\right],
\ee
where
\be
\begin{aligned}
j^{s,h}(t,z)
&=
\sum_{i\in \mathbb{T}_N}\sum_{k\ge1}
\Bigl[
c^{s,h}_{i,i+1,k}(t,z)
\log\frac{c^{s,h}_{i,i+1,k}(t,z)}
{c^s_{i,i+1,k}(z)}
+
c^s_{i,i+1,k}(z)
-
c^{s,h}_{i,i+1,k}(t,z)
\Bigr]
\\
&\quad
+
\sum_{i\in \mathbb{T}_N}\sum_{k\ge1}
\Bigl[
c^{s,h}_{i+1,i,k}(t,z)
\log\frac{c^{s,h}_{i+1,i,k}(t,z)}
{c^s_{i+1,i,k}(z)}
+
c^s_{i+1,i,k}(z)
-
c^{s,h}_{i+1,i,k}(t,z)
\Bigr].
\end{aligned}
\ee

On the localized set $K_M$, the rates are uniformly bounded and the logarithmic denominators stay bounded away from zero; therefore, by the same argument used for $H^s \to H$, one has $j^{s,h} \to j^h$ uniformly on $B(\bm{x}, \delta)$. Thus 
we obtain
\be
\lim_{s\to\infty}
\frac1s
\mathbb{E}_{\mathbb{P}^{s,h}_{x(0)}}\left[
\log \frac{d\mathbb{P}^s(\,\cdot\,|X^s(0))}
{d\mathbb{P}^{s,h}(\,\cdot\,|X^s(0))}
\right]
=
-\int_0^T j^h(t,x(t))\,dt .
\ee
where \(j^h\) denotes the pointwise limit of \(j^{s,h}\), namely
\be
\begin{aligned}
j^h(t,z)
&=
\sum_{i\in T_N}
\left[
\frac{z_i e^{h_{i+1}(t)-h_i(t)}}{1-z_i(e^{h_{i+1}(t)-h_i(t)}-1)}
\bigl(h_{i+1}(t)-h_i(t)\bigr)
+\log\!\left(1-z_i(e^{h_{i+1}(t)-h_i(t)}-1)\right)
\right]\\
&\quad+
\sum_{i\in T_N}
\left[
\frac{z_{i+1} e^{h_i(t)-h_{i+1}(t)}}{1-z_{i+1}(e^{h_i(t)-h_{i+1}(t)}-1)}
\bigl(h_i(t)-h_{i+1}(t)\bigr)
+\log\!\left(1-z_{i+1}(e^{h_i(t)-h_{i+1}(t)}-1)\right)
\right].
\end{aligned}
\ee
The standard convex duality relation between the Hamiltonian
$H$ and the entropy integrand $j^h$ gives
\be
\inf_{h:\,\dot{\bm{x}}=F^h(\bm{x})}
\int_0^T j^h(t,x(t))\,dt
=
I_{\mathrm{dyn}}(\bm{x}).
\ee
Therefore, for every admissible $h$ such that $\dot{\bm{x}}=F^h(\bm{x})$,
\be
\liminf_{\delta\downarrow0}\liminf_{s\to\infty}
\frac1s \log \mathbb{P}^s_\theta({\bm X}^s\in B(\bm{x},\delta))
\ge
-I_0(x(0))
-
\int_0^T j^h(t,x(t))\,dt .
\ee
Taking the infimum over such $h$, we obtain
\be
\liminf_{\delta\downarrow0}\liminf_{s\to\infty}
\frac1s \log \mathbb{P}^s_\theta({\bm X}^s\in B(\bm{x},\delta))
\ge
-I_0(x(0))-I_{\mathrm{dyn}}(\bm{x})
=
-I(\bm{x}).
\ee

\end{proof}

\subsection{Model with continuous state space}

\subsubsection{Brownian energy process}
Let $\Omega=(0,\infty)^{\mathbb{T}_N}$.
The Brownian energy process  $ \bm{X}^s$, taking values in $\Omega$,
is the diffusion process defined by the generator:
\begin{eqnarray}
\label{gen-bep}
({\cal L}^{s}f)({x}) 
&=& 
\sum_{i\in {\mathbb{T}_N}} -(x_i-x_{i+1})\left(\frac{\partial}{\partial x_i}-\frac{\partial}{\partial x_{i+1}}\right) +\frac{1}{s}x_ix_{i+1}\left(\frac{\partial}{\partial x_i}-\frac{\partial}{\partial x_{i+1}}\right)^2f(x).\nonumber\\
\end{eqnarray}
For a discussion of the Brownian energy process process and its main properties, see \cite{gkr2009}.
We denote by 
$\mathbb{P}^s_{\theta}$ the law of the process $\bm{X}^s$ initialized from configurations that are distributed according to $\mu^s_\theta$
having probability density function
\be
\label{initial-int-heat}
\mu^s_{\theta}(dx) = \bigotimes_{i\in {\mathbb T}_N}\frac{s^s}{\theta_i^s\Gamma(s)} x_i^{s -1 } e^{-\frac{s x_i}{\theta_i}} dx_i.
\ee
The marginal at each site $i$ is a (scaled) gamma distribution with mean $\theta_i>0$.
The homogeneous product measure $\mu_{\bar{\mathscr{\theta}}}$, with $\bar{\theta}$ having all components equal, is reversible for the Brownian energy process.
The measure \eqref{initial-int-heat}
describes a ``local equilibrium'' initial condition, corresponding to an inhomogenous version of the reversible measure.

The next theorem provides the path space large deviations of the Brownian energy process, which arises as a diffusive scaling limit of the inclusion process. In contrast with the previous models, the dynamics is governed by a second-order generator, leading to a quadratic Hamiltonian and requiring the use of Girsanov’s theorem for diffusion processes.
At the level of hydrodynamics, the large-deviation structure of the Brownian energy process have already
been studied in \cite{PeletierRedigVafayi2014}, where the resulting
rate functional is related to a gradient-flow structure for heat conduction. We
here present the large deviations around the fluid limit.

\begin{theorem}[LDP, Brownian energy process]
The sequence of Brownian energy processes $(\bm{X}^s)_{s\in\mathbb{N}}$, started from the distribution $(\mu^s_\theta)_{s\in \N}$ in \eqref{initial-int-heat}, satisfies a large deviation principle on $C([0,T],\Omega)$ with speed $s$ and rate function   
\begin{equation}
I(\bm{x})=I_0(x(0))+I_{\mathrm{dyn}}(\bm{x}),
\end{equation}
where the static part is given by
\begin{equation}
\label{static-I0-bep}
I_0(x)
=
\begin{cases}
\sum_{i\in\mathbb{T}_N}
\left[
-\log\frac{x_i}{\theta_i}
+\frac{x_i}{\theta_i} -1
\right], & \text{if } x_i>0 \quad \forall i,\\
+\infty & \text{otherwise}.
\end{cases}
\end{equation}
and the dynamical part is given by
\begin{equation}
I_{\mathrm{dyn}}(\bm{x})
=
\sup_{h\in C^1([0,T];\mathbb R^{\mathbb T_N})}
\left[
h(T)\cdot x(T)-h(0)\cdot x(0)
-
\int_0^T\bigl(\dot h(t)\cdot x(t)-H(x(t),h(t))\bigr)\,dt
\right],
\end{equation}
with Hamiltonian
\begin{equation}
H(x,w)
=
\sum_{i\in\mathbb T_N}
\Bigl(
(x_i-x_{i+1})(w_{i+1}-w_i)
+
x_i x_{i+1}(w_{i+1}-w_{i})^2
\Bigr).
\label{ham-bep}
\end{equation}
\end{theorem}

\begin{proof}
The proof follows the same strategy as that of the processes with discrete state space. We only highlight the main differences.
For the upper bound, the exponential martingale is constructed as in Section \ref{sec-upper}. Since the generator now contains a second-order term, a direct computation shows that
\be
\frac{1}{s} e^{-s h \cdot x} {\cal L}^s e^{s h \cdot x}
=
H(x,h),
\ee
where $H$ is given in \eqref{ham-bep}. This replaces the discrete jump expansion used in the harmonic and inclusion cases. The rest of the argument (tilting, optimization over $h$ and $a$, and exponential tightness) proceeds exactly as in Theorem \ref{main-theorem-harmonic}. Since the state space is non-compact, localization is performed using conservation of the total mass, as in the inclusion and harmonic cases.
For the lower bound,
we follow the same change-of-measure argument as in Section \ref{sec-lower}. 
The Radon--Nikodym derivative between the original and tilted processes is computed using Girsanov’s theorem for diffusion processes. 
\end{proof}

\begin{remark}
The Brownian energy process arises as a scaling limit of the inclusion process. More precisely, consider the
process ${\bm{\eta}^s} = \{\eta^s(t), t \ge 0\}$ with state space $\N^N$ defined by the generator 
\begin{eqnarray}
({L}f)(\eta) 
&=& 
\sum_{i\in {\mathbb{T}_N}} \eta_i(s+\eta_{i+1}) (f(\eta^{i,i+1})-f(\eta))\nonumber\\
&+& 
\sum_{i\in {\mathbb{T}_N}} \eta_{i+1}(s+\eta_i)(f(\eta^{i+1,i})-f(\eta)),\qquad\qquad \eta\in\N^N.
\end{eqnarray}
Here $\eta^{i,i+1}$ is  the configuration that is obtained from $\eta$ by moving one particle from $i$ to $i+1$.
The inclusion process we analyzed in Section \ref{sec:inclusion} is obtained considering $x=\eta/s$ and speeding up time by a factor $1/s$. 
If we instead introduce a scaling parameter $\epsilon>0$ and consider the  process $\bm{X}^{\epsilon,s}$ defined by ${X}_i^{\epsilon,s}(t)=\epsilon{\bm{\eta}(t/s)}$, then in the scaling limit $\epsilon \to 0$ we find  the process with generator \eqref{gen-bep}, i.e. ${\bm X}^s = \lim_{\epsilon\to 0} \bm{X}^{\epsilon,s}$ (weak convergence). 
Indeed the process $\bm{X}^{\epsilon,s}$
has generator
\begin{eqnarray}
({\cal L}^{\epsilon,s}f)({x}) 
&=& 
\frac{1}{s}\sum_{i\in {\mathbb{T}_N}} \epsilon^{-1}x_i(s+\epsilon^{-1}x_{i+1}) (f(x^{i,i+1})-f(x))\nonumber\\
&+& 
\frac{1}{s}\sum_{i\in {\mathbb{T}_N}} \epsilon^{-1}x_{i+1}(s+\epsilon^{-1}x_i)(f(x^{i+1,i})-f(x)),\qquad\qquad x\in\Omega^s,
\end{eqnarray}
and, by a Taylor expansion, one gets $({\cal L}^{\epsilon,s}f)({x})\to ({\cal L}^{s}f)({x})$
as $\epsilon\to 0$, uniformly in $x$.
This remark explains
why a naive expansion 
up to second order 
of \eqref{7.24} 
differs from the expression
we got in \eqref{ham-bep}. 
\end{remark}

\subsubsection{Integrable heat conduction model}

Let $\Omega=(0,\infty)^{\mathbb{T}_N}$.
The integrable heat conduction process  $ \bm{X}^s$, taking values in $\Omega$,
is the Markov process defined by the generator \cite{fgk}:
\begin{eqnarray}
\label{gen-int-heat}
({\cal L}^{s}f)({x}) 
&=& 
s \sum_{i\in {\mathbb{T}_N}} \int_{0}^{s x_i } \frac{du}{u}\Big(1-\frac{u}{sx_i}\Big)^{s-1}
(f(x^{i,i+1,u})-f(x))\nonumber\\
&+& 
s \sum_{i\in {\mathbb{T}_N}} \int_{0}^{ s x_{i+i}}\frac{du}{u} \Big(1-\frac{u}{sx_{i+1}}\Big)^{s-1}(f(x^{i+1,i,u})-f(x)), \qquad\qquad x\in\Omega.
\end{eqnarray}
Here, $x^{j,\ell, u}$ denotes the configuration that is obtained from $x\in\Omega$ by moving a mass $\frac{u}{s}$ from site $j$ to site $\ell$, i.e.
\begin{eqnarray}
x^{j,\ell,u}_i =
\begin{cases}
x_i, &  i\neq j,\ell,\\
x_i-\frac{u}{s},\qquad &  i = j,\\
x_i+\frac{u}{s}, & i = \ell.
\end{cases}
\end{eqnarray}
For a discussion of the integrable heat conduction process and its main properties, see \cite{fgk}.

\begin{theorem}[LDP, integrable heat conduction process]
\label{main-theorem-int-heat}
The sequence of intergable heat conduction processes $(\bm{X}^s)_{s\in\mathbb N}$, 
 initialized from the distribution $(\mu^s_{\theta})_{s\in\mathbb N}$ in \eqref{initial-int-heat},  satisfies a large deviation principle on $D([0,T],\Omega)$ with speed $s$ and rate function
\begin{equation}
I(\bm{x})=I_0(x(0))+I_{\mathrm{dyn}}(\bm{x}),
\end{equation}
where the initial cost is given by
\begin{equation}
\label{static-I0-int-heat}
I_0(x)
=
\begin{cases}
\sum_{i\in\mathbb{T}_N}
\left[
-\log\frac{x_i}{\theta_i}
+\frac{x_i}{\theta_i} -1
\right], & \text{if } x_i>0 \quad \forall i,\\
+\infty & \text{otherwise}.
\end{cases}
\end{equation}
and the dynamical part is given by
\begin{equation}
I_{\mathrm{dyn}}(\bm{x})
=
\sup_{h\in C^1([0,T];\mathbb{R}^{T_N})}
\left[
h(T)\cdot x(T)-h(0)\cdot x(0)
-
\int_0^T\bigl(\dot h(t)\cdot x(t)-H(x(t),h(t))\bigr)\,dt
\right],
\label{I-dyn-int-heat}
\end{equation}
with Hamiltonian
\begin{equation}
H(x,w)
=
-\sum_{i\in\mathbb T_N}
\log\Bigl(1-x_i(w_{i+1}-w_i)\Bigr)
-\sum_{i\in\mathbb T_N}
\log\Bigl(1-x_{i+1}(w_i-w_{i+1})\Bigr),
\label{ham-heat}
\end{equation}
with the convention that $H(x,w)=+\infty$ whenever one of the logarithms
is not well defined.
\end{theorem}

\begin{proof}
For the upper bound, the exponential martingale is
constructed as in Section \ref{sec-upper}. In the present continuous-state jump process the
finite sums over the jump sizes are replaced by integrals over \(u\). A direct
computation gives
\be
\frac1s e^{-s h\cdot x}{\cal L}^s e^{s h\cdot x}=H^s(x,h),
\ee
where
\begin{eqnarray}
H^s(x,w)&=&
\sum_{i\in T_N}\int_0^{sx_i}\frac{du}{u}
\left(1-\frac{u}{sx_i}\right)^{s-1}
\left(e^{u(w_{i+1}-w_i)}-1\right)\nn\\
&+&
\sum_{i\in T_N}\int_0^{sx_{i+1}}\frac{du}{u}
\left(1-\frac{u}{sx_{i+1}}\right)^{s-1}
\left(e^{u(w_i-w_{i+1})}-1\right).
\end{eqnarray}
As \(s\to\infty\), using
\be
\left(1-\frac{u}{sx_i}\right)^{s-1}\longrightarrow e^{-u/x_i},
\ee
one obtains, uniformly on localized sets where the logarithmic arguments are
bounded away from zero,
\be
H^s(x,w)\longrightarrow H(x,w),
\ee
with $H$ given in \eqref{ham-heat}. The rest of the upper bound, including the
optimization over \(h\) and over the initial exponential tilt, is the same as in
Theorem \ref{main-theorem-harmonic}. The non-compactness of the state space is handled by localization
through conservation of the total mass.
For the lower bound, one uses the tilted version of the process obtained by
multiplying the jump kernel by
\be
e^{u(h_{i+1}(t)-h_i(t))}
\qquad\text{and}\qquad
e^{u(h_i(t)-h_{i+1}(t))}
\ee
for jumps $i\to i+1$ and $i+1\to i$, respectively. The change-of-measure
formula is the Girsanov formula for jump processes with continuous jump
parameter $u$; thus the sums over $k$ appearing in the proof of
Theorem \ref{main-theorem-harmonic} are replaced by integrals over $u$. 
\end{proof}

\bibliographystyle{abbrv}
\bibliography{ref}




\end{document}